\documentclass[11pt,fleqn]{amsart}

\usepackage[small]{titlesec}
\usepackage{hyperref}
\usepackage{amsthm,thmtools}
\usepackage[utf8]{inputenc}
\usepackage[english]{babel}
\usepackage{enumerate}
\usepackage[osf,noBBpl]{mathpazo}
\usepackage[alphabetic,initials]{amsrefs}
\usepackage{amsfonts,amssymb,amsmath}
\usepackage{mathtools}
\usepackage{graphicx}
\usepackage[poly,arrow,curve,matrix]{xy}
\usepackage{helvet}
\usepackage{stmaryrd}
\usepackage{tikz}
\usepackage{mathdots}

\usetikzlibrary{arrows.meta}
\tikzset{>={Latex[width=1mm,length=1.2mm]}}

\renewcommand\thesection{\arabic{section}}
\titleformat{\section}
 {\normalfont\bfseries\large}
 {\thesection. \space}{0em}{}

\renewcommand\thesubsection{\arabic{subsection}}
\titleformat{\subsection}
 {\normalfont\bfseries}
 {\thesection.\thesubsection \space}{0em}{}

\renewcommand\proofname{Proof}

\makeatletter
\renewenvironment{proof}[1][\textit{\proofname}]{\par
 \pushQED{\qed}%
 \normalfont \topsep.75\paraskip\relax
 \trivlist
 \item[\hskip\labelsep
 \itshape
 #1\@addpunct{.}]\ignorespaces
}{%
 \popQED\endtrivlist\@endpefalse
}
\makeatother
\declaretheoremstyle[
bodyfont=\itshape,]{mystyle}
\declaretheorem[name=Lemma, style=mystyle, numberwithin=section]{Lemma}
\declaretheorem[name=Proposition, style=mystyle, sibling=Lemma]{Proposition}
\declaretheorem[name=Theorem, style=mystyle, sibling=Lemma]{Theorem}
\declaretheorem[name=Corollary, style=mystyle, sibling=Lemma]{Corollary}
\declaretheorem[name=Definition, style=mystyle, sibling=Lemma]{Definition}
\declaretheorem[name=Example, style=mystyle, sibling=Lemma]{Example}
\declaretheorem[name=Remark, style=mystyle, sibling=Lemma]{Remark}

\declaretheoremstyle[numbered=no, 
bodyfont=\itshape]{mystyle-empty}
\declaretheorem[name=Lemma, style=mystyle-empty]{Lemma*}
\declaretheorem[name=Proposition, style=mystyle-empty]{Proposition*}
\declaretheorem[name=Theorem, style=mystyle-empty]{Theorem*}
\declaretheorem[name=Corollary, style=mystyle-empty]{Corollary*}
\declaretheorem[name=Definition, style=mystyle-empty]{Definition*}
\declaretheorem[name=Example, style=mystyle-empty]{Example*}
\declaretheorem[name=Remark, style=mystyle-empty]{Remark*}

\newskip\paraskip
\newcounter{para}[section]
\renewcommand\thepara{\thesection.\arabic{para}}
\def\paragraph{%
\vspace{1ex}
 \noindent
 \refstepcounter{para}%
 \textbf{\thepara.}\hspace{1ex}%
}

\newcommand\about[1]{%
 {\bfseries#1.}%
}

\newcommand\NN{\mathbb N}
\newcommand\CC{\mathbb C}
\newcommand\QQ{\mathbb Q}
\newcommand\RR{\mathbb R}
\newcommand\ZZ{\mathbb Z}

\renewcommand\to{\rightarrow}
\renewcommand\phi{\varphi}

\newcommand\h{\mathfrak h}
\newcommand\m{\mathfrak m}
\newcommand\gl{\mathfrak{gl}}
\renewcommand\sl{\mathfrak{sl}}
\newcommand\vv{\overline{v}}
\newcommand\II{\mathbb I}

\newcommand\GT{\mathsf{GT}}
\renewcommand\P{\mathcal P}

\newcommand\interval[1]{\llbracket #1 \rrbracket}
\newcommand\abs[1]{|#1|}

\DeclareMathOperator\Specm{Specm}

\DeclareMathOperator\ess{ess}

\DeclareMathOperator\supp{supp}
\DeclareMathOperator\essupp{essupp}

\newcommand\DD{\mathbb D}
\newcommand\D{\mathfrak D}

\newcommand\soc{\mathrm{soc}}

\title{Bounds of Gelfand-Tsetlin multiplicities and tableaux realizations of 
Verma modules}
\date{}

\author[V. Futorny]{Vyacheslav Futorny}
\address{Instituto de Matem\'atica e Estat\'istica, Universidade de S\~ao
Paulo, S\~ao Paulo SP, Brasil} \email{futorny@ime.usp.br,}
\author[D. Gratcharov]{Dimitar Grantcharov}
\address{\noindent
University of Texas at Arlington, Arlington, TX 76019, USA} \email{
grandim@uta.edu}
\author[L.E. Ramirez]{Luis Enrique Ramirez}
\address{Universidade Federal do ABC, Santo Andr\'e-SP, Brasil} \email{
luis.enrique@ufabc.edu.br,}
\author[P. Zadunaisky]{Pablo Zadunaisky}
\address{Universidad CAECE and Universidad de Buenos Aires 
-- Buenos Aires, Argentina} \email{pzadunaiskybustillos@caece.edu.ar}

\begin{document}
\begin{abstract}
We introduce the notion of essential support of a simple Gelfand-Tsetlin $\gl_n$-module as an attempt towards understanding the character formula of such module. This support detects the weights in the module having maximal possible Gelfand-Tsetlin multiplicities. Using combinatorial tools we describe the essential supports of the simple socles of the universal tableaux modules. We also prove that every simple Verma module appears as the socle of a universal tableaux module. As a consequence, we prove the Strong Futorny-Ovsienko Conjecture on the sharpness of the upper bounds of the Gelfand-Tsetlin multiplicities. We also give a very explicit description of the support and essential support of the simple singular Verma module $M(-\rho)$.

\end{abstract}

\maketitle


%\vspace{-2cm}

\section{Introduction}

Gelfand-Tsetlin modules of the complex general linear Lie algebra $\gl(n, \CC)$ 
have been studied since the 1950's both by mathematicians and physicists. 
These modules admit a locally finite action of the Gelfand-Tsetlin subalgebra 
$\Gamma$, a maximal commutative subalgebra of the universal enveloping algebra 
$U(\gl(n))$. The recent discovery of Gelfand-Tsetlin derivative tableaux 
in \cite{FGR16} initiated the systematic study of singular Gelfand-Tsetlin 
modules. This theory attracted considerable attention in the last three years 
and many interesting and important results have been obtained in \cites{EMV18, 
FGR17a, FGR17b, FGRZ18, FRZ16a, Hartwig18, MV18, RZ18, Vis17, Vishnyakova18, Zad-1-sing}. Singular
Gelfand-Tsetlin modules turned out to be related to different but overlapping 
theories. For example, connections with Schubert calculus were discovered 
in \cite{FGRZ18} and with tensor product categorifications and 
KLRW algebras in \cite{KTWWY18}. With the aid of KLRW algebras, in \cite{KTWWY18}, the authors provide a bijection between the set of simple Gelfand-Tsetlin $\gl(n, \CC)$-modules with a fixed character and the zero weight space of an $\mathfrak{sl}(n, \CC)$-crystal.
Furthermore, the properties of the singular Gelfand-Tsetlin modules have been 
studied with combinatorial tools, \cites{FGR16, FGRZ18, FRZ16a, RZ18}, as 
well as with geometric methods, \cites{EMV18, MV18, Vishnyakova18}. 

A maximal ideal $\m$ of $\Gamma$ defines a point $v = v_\m$ in $ \mathbb 
C^{\mu}=\mathbb C^{1}\times \mathbb C^{2}\times \cdots \times \mathbb C^{n}$ 
up to a permutation of coordinates. Given such $v$ one constructs a  
\emph{"universal" tableaux Gelfand-Tsetlin module} $V(T(v))$ which contains a 
simple Gelfand-Tsetlin subquotient $M$ having $\m$ in its support, i.e. $M[\m]
\neq 0$ (see \S\ref{subsec-gt-def} for precise definitions). It was shown in 
\cites{EMV18,RZ18} that $V(T(v))$ has a basis of \emph{derivative tableaux}, 
and the action of $\gl(n,\CC)$ on this basis was described in \cite{FGRZ18} 
in terms of BGG differential operators and Postnikov-Stanley polynomials. We 
conjecture that the module $V(T(v))$ is universal in the sense that every 
simple Gelfand-Tsetlin module having $\m$ in its support is a subquotient of 
$V(T(v_\m))$. This conjecture was proven for generic $v$ in \cite{FGR15} and 
for $1$-singular $v$ in \cites{FGR16, FGR17b}.

In the present paper we make a significant step in the understanding of the 
structure of $V(T(v))$, in particular its socle. As a generating vector we 
choose a special vector $\vv\in \mathbb C^{\mu}$, called a \emph{seed} (see 
Definition \ref{D:seed}), and show that
the module $V(T(v)) = V(T(\vv))$ has a simple socle $V_{\soc}$ whose structure 
can be described in terms of certain oriented graphs. The simple module 
$V_{\soc}$ is also a Gelfand-Tsetlin module such that $V_{\soc} =\bigoplus_{z} 
V_{\soc} [\vv+z]$ where the sum is taken over a certain set of points of 
$\mathbb C^{\mu}$ with integral coordinates. The dimensions of the weight 
spaces $V_{\rm soc} [\vv+z]$ (called Gelfand-Tsetlin multiplicities) are 
finite and uniformly bounded as explained in more detail below.

Set $S_{\mu}:=S_1\times \cdots \times S_n$ and consider the free abelian 
group $\ZZ_0^\mu$ consisting of elements in $\mathbb C^{\mu}$ with integer 
coordinates the last $n$ of which equal  zero. Denote by $\GT$ 
the category of all Gelfand-Tsetlin $\gl(n,\CC)$-modules, and for each equivalence class 
$\zeta \in \CC^\mu / (\ZZ_0^\mu \# S_\mu)$ denote by $\GT_\zeta$ the full 
subcategory of $\GT$ consisting of modules whose support is contained in 
$\zeta$. We have a decomposition of $\GT$ into a direct sum of components
\begin{align*}
\GT &= \bigoplus_{\zeta \in \CC^\mu / (\ZZ_0^\mu \# S_\mu)} \GT_\zeta
\end{align*}
in the sense that $\operatorname{Ext}^i_\GT(M,N) = 0$ for all $i \geq 0$ and 
for any $M$ and $N$ in different components (see \cite{FO14}*{Corollary 3.4}). 

An upper bound for the Gelfand-Tsetlin multiplicities of any simple 
Gelfand-Tsetlin module was found in \cite{FO14}*{Theorem 4.12(c)}. To write this bound, 
fix a seed $\vv$  and consider the stabilizer $S_{\pi(\vv)}$ of 
$\vv$ in $ S_\mu$. Let $z\in\ZZ_0^\mu$ be such that $\vv+z$ is in normal form 
(see Definition \ref{def-nor-form}) and let $(S_{\pi(\vv)})_z$ be the 
stabilizer of $z$ in $S_{\pi(\vv)}$. Set $\zeta=\zeta_{\vv} =\CC^\mu / (\ZZ_0^\mu \# S_\mu) 
\vv$.  Then, as shown in \cite{FO14}, for any simple 
module $M$ in $\GT_\zeta$ the upper bound on a Gelfand-Tsetlin multiplicity
is given by
\begin{equation}\label{FO-inequality} 
\dim M[\vv+z]
\leq \frac{|S_{\pi(\vv)}|}{|(S_{\pi(\vv)})_z|}.
\end{equation} 
We will refer to (\ref{FO-inequality}) as the \emph{FO inequality}. In 
\cite{FO14}*{Remark 5.4} Futorny and Ovsienko conjectured that this inequality
is sharp, and more precisely, that there is a simple module $M$ for which
equality holds in (\ref{FO-inequality}) for some  $z$ with trivial 
stabilizer. This conjecture follows from either \cite{FGRZ18}*{Theorems 8.3, 
8.5} or \cite{EMV18}*{Theorems 10,11} for $z$ with any stabilizer.
In fact, the results in \cite{EMV18} and \cite{FGRZ18} imply that 
(\ref{FO-inequality}) holds for $M=V(T(\vv))$. In the special case where none 
of the differences between entries in consecutive rows of $\vv$ are integers, 
$V(T(\vv))$ is simple. 

In this paper we prove a stronger result, which we call the \emph{Strong 
Futorny-Ovsienko Conjecture}. We show that the socle of $V(T(\vv))$ is 
simple, and that taking $M$ to be this socle, the set of all $z \in \ZZ^\mu_0$ 
such that equality holds in (\ref{FO-inequality}) (the \emph{essential support}
of $M$) is a union of rational polyhedral cones, at least one of which has
dimension $n(n-1)/2$. We also show that the subgroups $(S_{\pi(\vv)})_z$ 
run over all parabolic subgroups of $S_{\pi(\vv)}$. This shows that the FO 
inequality gives a sharp bound in each subcategory $\GT_\zeta$.

In the following theorem we summarize the above discussed results. 
\begin{Theorem} 
Let $\vv$ be a seed in $\mathbb C^{\mu}$ and $\zeta=\zeta_{\vv}$. Then the 
following hold. 
\begin{itemize}
\item[(i)]
The module $V(T(\vv))$ has a simple socle $V_{\rm soc}$.

\item[(ii)] 
The Strong Futorny-Ovsienko Conjecture holds for $V_{\rm soc}$, i.e. the 
essential support of $V_{\rm soc} $ is nonempty. Moreover, this essential 
support consists of the integral points of a finite union of polyhedral 
rational cones, at least one of which is of maximal possible rank, 
$\frac{n(n-1)}{2}$.

\item[(iii)] 
The maximal Gelfand-Tsetlin multiplicity of a character in $\GT_\zeta$ is
$|S_{\pi(\vv)}|$, and this is attained at the socle $V_{\soc}$. 

\item[(iv)] 
For any $\vv+z$ in the essential support of $V_{\soc}$, the module 
$V_{\rm soc}$ is the unique simple Gelfand-Tsetlin module having $\vv+z$ in 
its support.

\item[(v)] 
For any parabolic subgroup $G \subset S_{\pi(\vv)}$, the quotient 
$\frac{|S_{\pi(\vv)}|}{|G|}$ appears as a Gelfand-Tsetlin multiplicity of the 
module $V_{\soc}$. 
\end{itemize}
\end{Theorem}

In Section 6 we apply the above results to study the Gelfand-Tsetlin structure 
of Verma modules. Theorems \ref{thm-supp Verma} and \ref{thm-esssup Verma} 
describe the support and the essential support of the simple singular Verma 
module $M(-\tilde\rho)$, respectively, where $\tilde \rho = -(0, 1, \ldots, 
n-1)$; notice that as $\sl(n,\CC)$-module, this is isomorphic to the Verma 
module associated to minus the half-sum of the positive roots. Our second main 
result is summarized below.

\begin{Theorem} 
\begin{itemize}
\item[(i)] 
Every Verma module is a submodule of a certain universal tableaux module 
$V(T(\vv))$, and every simple Verma module appears as the socle of some 
$V(T(\vv))$. In particular, $M(-\tilde\rho)$ is the socle of $V(T(\mathbf 0))$.

\item[(ii)] 
The essential support of $M(-\tilde\rho)$ is a rational cone. It contains 
weights with Gelfand-Tsetlin multiplicities $\displaystyle 
\frac{\Pi_{i=1}^{n-1}i!}{|G|}$ for any standard parabolic subgroup $G$ of 
$S_1\times \ldots \times S_{n-1}$.

\item[(iii)] 
The maximal Gelfand-Tsetlin multiplicity in $M(-\tilde\rho)$ is 
$\Pi_{i=1}^{n-1}i!$.
\end{itemize}
\end{Theorem}

The paper ends with a counterexample to the ``hope'' that the simple 
subquotients of $V(T(\vv))$ have bases consisting of derivative tableaux. 
We show that this is not the case even for the simple Verma $\gl(4,
\CC)$-module $M(-\tilde\rho)$.

\medskip

\noindent{\bf Acknowledgements.} We thank Ben Webster for pointing out that a conjecture in an earlier version of the paper was not true. V.F. is supported  by CNPq grants
(304467/2017-0 and 200783/2018-1). D.G. is supported in part by Simons Collaboration Grant 
358245. L.E.R. is supported by Fapesp grant (2018/17955-7). P.Z. is supported by a CONICET 
postdoctoral fellowship.

\section{Preliminaries on the combinatorics of Gelfand-Tsetlin tableaux}

\paragraph
\about{Notation}
\label{notation} 
Given $a,b,k \in \NN$ we set $\interval{a,b} = \{i \in \NN \mid a \leq i \leq 
b\}$, with $\interval{b} = \interval{1,b}$; also we denote by $S_k$ the 
symmetric group in $k$ elements. Given $\pi = (\pi_1, \ldots, \pi_r) \in \NN^r$
with $\sum_i \pi_i = k$ we denote by $S_\pi$ the product $S_{\pi_1} \times 
S_{\pi_2} \times \cdots \times S_{\pi_r}$, which we see as a subgroup of $S_k$.

Fix $n \in \NN$ and let $\mu = (1, 2, \ldots, n)$. Given $\sigma \in S_\mu$ 
and $k \in \interval n$ we denote by $\sigma_{(k)}$ the projection of $\sigma$
to $S_k$. With a slight abuse of notation, we identify $S_k$ with the subgroup 
of $S_\mu$ consisting of elements $\sigma$ such that $\sigma_{(i)}$ is the 
identity for all $i \neq k$. Thus we can write $\sigma = \sigma_{(1)} 
\sigma_{(2)} \cdots \sigma_{(n)}$.

The group $S_\mu$ is a Coxeter group with generating set 
$$
\{(i \ i+1)_{(k)} \mid k \in \interval{n}, i \in \interval{k-1}\},
$$
where $(i \ i+1)$ is the simple transposition interchanging $i$ and $i+1$. The 
usual notions of length,
Bruhat order, parabolic subgroups, etc. will be considered with respect to 
this generating set. In particular, the length of $\sigma \in S_\mu$ is 
 $\ell(\sigma) = \ell_1(\sigma_{(1)}) + \cdots + \ell_n(\sigma_{(n)})$, 
where $\ell_k$ stands for the usual length in $S_k$. Also if $\tau \in S_\mu$ 
then $\sigma < \tau$ in the Bruhat order if and only if $\sigma_{(k)} < 
\tau_{(k)}$ for all $k$.

Henceforth we fix $\Sigma = \{(k,i) \mid 1 \leq i \leq k \leq n\}$. The group 
$S_\mu$ acts on $\Sigma$ with the action given by $\sigma \cdot (k,i) = (k, 
\sigma_{(k)}(i))$. The subset $\Sigma' = \{(k,i) \mid 1 \leq i \leq k \leq n-1
\}$ is clearly invariant under this action. 

For $k\geq a,b$ we set $\interval{a,b}_k = \{(k,i) \mid i \in \interval{a,b}\} 
\subset \Sigma$. Such a set will be called an \emph{interval}
of $\Sigma$, and given an interval $I = \interval{a,b}_k$ we write $a(I) = a,
b(I) = b, k(I) = k$. A partition of $\Sigma$ is a family of nonempty subsets 
of $\Sigma$, which we call \emph{blocks}, whose disjoint union is $\Sigma$.
An \emph{interval partition} is a partition $\II$ whose blocks are intervals. 
We write $\II[k]$ for the set of all intervals $I \in \II$ with $k(I) = k$. 

\paragraph
Let $\CC^\mu = \CC \times \CC^2 \times \cdots \times \CC^n$, so every $v \in 
\CC^\mu$ is an $n$-tuple $(v_1, v_2, \ldots, v_n)$ with $v_k = (v_{k,1}, 
v_{k,2}, \ldots, v_{k,k}) \in \CC^k$. For each $(k,i) \in \Sigma$ we denote by 
$\delta^{k,i}$ the unique element in $\CC^\mu$ such that 
$(\delta^{k,i})_{l,j} = \delta_{k,l} \delta_{i,j}$, and refer to the set 
$\{\delta^{k,i}\mid (k,i) \in \Sigma\}$ as the canonical basis of $\CC^\mu$. 
The group $S_\mu$ acts on $\CC^\mu$ by linear operators whose action on the 
canonical basis is given by $\sigma \cdot \delta^{k,i} = \delta^{\sigma \cdot 
(k,i)} = \delta^{k, \sigma_{(k)} (i)}$. We denote by $\ZZ^\mu_0$ the additive 
subgroup of $\CC^\mu$ generated by $\{\delta^{k,i} \mid (k,i) \in \Sigma'\}$, 
which is stable under the action of $S_\mu$.

Given an interval $I = \interval{a,b}_k$ we will write $v(I)$ for $(v_{k,a}, 
v_{k,a+1}, \ldots, v_{k,b})$. Given an interval partition $\II$ of $\Sigma$ 
we refer to the tuples $v(I)$ with $I \in \II$ as the blocks of $v$. We 
associate to $v$ a partition of $\Sigma$ denoted by $\II(v)$, where the 
block of $(n,i)$ is $\{(n,i)\}$, and for $k < n$ the block of $(k,i)$ is the 
set of all $(k,j)$ such that $v_{k,i} - v_{k,j} \in \ZZ$.
\begin{Definition} \label{def-nor-form}
We say that $v \in \CC^\mu$ is in \emph{normal form} if whenever $v_{k,a} - 
v_{k,b} \in \ZZ$ for some $a\leq b \leq k \leq n$, then $v_{k,i} - v_{k,j} 
\in \ZZ_{\geq 0}$ for all $a \leq i < j \leq b$. 
\end{Definition}
Notice that if $v$ is in normal form then $\II(v)$ is an interval partition
(but not the other way around). Clearly for each $v \in \CC^\mu$ there exists 
at least one element in its $S_\mu$-orbit which is in normal form. 
Suppose $v$ is in normal form and let $\II(v)[k] = \{I_1, I_2, \ldots, I_r\}$, 
with $a(I_i) = b(I_{i-1}) + 1$. We set $\pi(v,k) = (\abs{I_1},\abs{I_2}, 
\ldots, \abs{I_r})$, so $S_{\pi(v,k)}$ is a parabolic subgroup of $S_k$; 
observe that by definition $S_{\pi(v,n)}$ is the trivial subgroup of $S_n$. 
We denote by $\pi(v)$ the concatenation of $\pi(v,1), \ldots, \pi(v,n)$, and 
set $S_{\pi(v)} = S_{\pi(v,1)} \times S_{\pi(v,2)} \times \cdots 
\times S_{\pi(v,n)} \subset S_\mu$. This is a parabolic subgroup of $S_\mu$.

\paragraph
\textbf{The graph $\Omega(v)$.}
We now associate to each element $v \in \CC^\mu$ a graph, which will be a 
major combinatorial tool in this paper.
Given $v \in \CC^\mu$ the graph $\Omega(v)$ is defined as follows: the set of 
vertices of $\Omega(v)$ is $\{[k,i] \mid (k,i) \in \Sigma\}$, and we have an 
edge between $[k,i]$ and $[l,j]$ if and only if $v_{k,i} - v_{l,j} \in \ZZ$ 
and $|k-l| \leq 1$. We will use the notation $[k,i] - [l,j]$ for an edge 
between $[k,i]$ and $[l,j]$. 

\begin{Definition}
\label{D:seed}
We say that $\vv \in \CC^\mu$ is a \emph{seed} if it is in normal form and for
$[k,i]$ and $[l,j]$ in the same connected component of $\Omega(v)$ the 
following holds: if $k,l < n$ then $\vv_{k,i} = \vv_{l,j}$, while if $l = n$
then $v_{k,i} \leq v_{n,j}$.
\end{Definition}

As mentioned before, $S_\mu$ acts on $\CC^\mu$ and $\ZZ_0^\mu$ is stable under
this action. Also, $\ZZ_0^\mu$ acts on $\CC^\mu$ by translations: $z \cdot v = 
v + z$ for $z \in \ZZ_0^\mu$ and $v \in \CC^\mu$. Thus the semidirect product 
$\ZZ_0^\mu \# S_\mu$ acts on $\CC^\mu$. If $v$ is in normal form then there 
exists $z \in \ZZ^\mu_0$ such that $v+z$ is a seed. Hence for every $v \in 
\CC^\mu$ there exists a seed in its $(\ZZ_0^\mu \# S_\mu)$-orbit. Two 
elements $v,w \in \CC^\mu$ lie in the same orbit if and only if there exists
$\sigma \in S_\mu$ such that $\sigma(\Omega(v)) = \Omega(w)$ and $v - 
\sigma(w) \in \ZZ^\mu_0$, where $\sigma$ acts on the vertices of the graph in 
the obvious way.

\begin{Example}
We will write elements of $\CC^\mu$ as triangular arrays with $k$ entries in
the $k$-th row counting from the bottom. In this example we assume that the 
set $\{1, a, b, c, \ldots\} \subset \CC$ is linearly independent over $\ZZ$.

\medskip
\begin{tabular}{cc}
\begin{tikzpicture}
\node (51) at (-2,2.5) {$1$};
\node (52) at (-1,2.5) {$a+1$};
\node (53) at (0,2.5) {$a$};
\node (54) at (1,2.5) {$b$};
\node (55) at (2,2.5) {$0$};

\node (41) at (-1.5,2) {$a$};
\node (42) at (-0.5,2) {$b-1$};
\node (43) at (0.5,2) {$b$};
\node (44) at (1.5,2) {$a+1$};

\node (31) at (-1,1.5) {$c$};
\node (32) at (0,1.5) {$c+1$};
\node (33) at (1,1.5) {$c$};

\node (21) at (-.5,1) {$a$};
\node (22) at (.5,1) {$a-1$};

\node (11) at (0,0.5) {$a+1$};

\node (text) at (0,-0.2) {An element $v$ of $\CC^\mu$};
\end{tikzpicture}

&\begin{tikzpicture}
\node (51) at (-2,2.5) {$1$};
\node (52) at (-1,2.5) {$0$};
\node (53) at (0,2.5) {$a+1$};
\node (54) at (1,2.5) {$a$};
\node (55) at (2,2.5) {$b$};

\node (41) at (-1.5,2) {$a+1$};
\node (42) at (-0.5,2) {$a$};
\node (43) at (0.5,2) {$b$};
\node (44) at (1.5,2) {$b-1$};

\node (31) at (-1,1.5) {$c+1$};
\node (32) at (0,1.5) {$c$};
\node (33) at (1,1.5) {$c$};

\node (21) at (-.5,1) {$a$};
\node (22) at (.5,1) {$a-1$};

\node (11) at (0,0.5) {$a+1$};

\node (text) at (0,-0.2) {An element in normal form in $S_\mu v$};
\end{tikzpicture} \\

\begin{tikzpicture}
\node (51) at (-2,2.5) {$a+1$};
\node (52) at (-1,2.5) {$a$};
\node (53) at (0,2.5) {$1$};
\node (54) at (1,2.5) {$0$};
\node (55) at (2,2.5) {$b$};

\node (41) at (-1.5,2) {$a$};
\node (42) at (-0.5,2) {$a$};
\node (43) at (0.5,2) {$b-1$};
\node (44) at (1.5,2) {$b-1$};

\node (31) at (-1,1.5) {$c$};
\node (32) at (0,1.5) {$c$};
\node (33) at (1,1.5) {$c$};

\node (21) at (-.5,1) {$a-1$};
\node (22) at (.5,1) {$a-1$};

\node (11) at (0,0.5) {$a-1$};

\node (text) at (0,-0.2) {A seed in $(\ZZ_0^\mu \# S_\mu) v$};
\end{tikzpicture}
&
\tikzstyle{place}=[circle,draw=black,fill=black,
   inner sep=1pt,minimum size=1mm]
\begin{tikzpicture}
\node (51) at (-2,2.5) [place] {};
\node (52) at (-1,2.5) [place] {};
\node (53) at (0,2.5) [place] {};
\node (54) at (1,2.5) [place] {};
\node (55) at (2,2.5) [place] {};

\node (41) at (-1.5,2) [place] {};
\node (42) at (-0.5,2) [place] {};
\node (43) at (0.5,2) [place] {};
\node (44) at (1.5,2) [place] {};

\node (31) at (-1,1.5) [place] {};
\node (32) at (0,1.5) [place] {};
\node (33) at (1,1.5) [place] {};

\node (21) at (-.5,1) [place] {};
\node (22) at (.5,1) [place] {};

\node (11) at (0,0.5) [place] {};

\node (text) at (0,-0.2) {The graph of the previous elements};

\draw (51) -- (52) -- (42) -- (41) -- (51) -- (42) (52) -- (41);
\draw (53) -- (54);
\draw (55) -- (43) -- (44) -- (55);
\draw (31) -- (32) -- (33) to[out=155,in=25] (31);
\draw (11) -- (21) -- (22) -- (11);
\end{tikzpicture}
\end{tabular}
\end{Example}

\paragraph
\label{descending-z}
\about{The set $\DD(\vv)$}
Recall that we denote by $\ZZ^\mu_0$ the set of all $z \in \CC^\mu$ with
$z_{k,i} \in \ZZ$ for all $(k,i) \in \Sigma$ and $z_{n,i} = 0$ for all $i \in
\interval n$.
\begin{Definition}
\label{D:normal-forms}
Let $\vv \in \CC^\mu$ be a seed. We denote by $\DD(\vv)$ the set of all
$z \in \ZZ_0^\mu$ such that $z(I)$ is a nonincreasing sequence for all $I \in 
\II(\vv)$.
\end{Definition}
For the rest of this section we fix a seed $\vv$ and set $\II = \II(\vv), \pi =
\pi(\vv)$ and $\DD = \DD(\vv)$. Notice that $\DD$ is the set of those
$z \in \ZZ^\mu_0$ such that $\vv + z$ is in normal form. For example, if
$\vv$ is the zero vector $\mathbf 0 \in \CC^\mu$ then $\DD(\mathbf 0)$ is the 
set of $z \in \ZZ^\mu_0$ such that $z_{k,1} \geq z_{k,2} \geq \cdots \geq 
z_{k,k}$ for all $k \in \interval{n-1}$.

Let $z \in \DD$. The stabilizer of $z$ in $S_\pi$ is again a parabolic 
subgroup of $S_\pi$, which we denote as usual by $(S_\pi)_z$, so each 
coclass in $S_\pi/(S_\pi)_z$ has a unique minimal length representative. We 
denote by $S_\pi^z$ the set of these minimal length representatives, and refer 
to them as \emph{$z$-shuffles}. Given $\sigma \in S_\pi$ we write $\sigma^z$ 
for the unique $z$-shuffle in $\sigma (S_\pi)_z$.

We denote by $\II(\vv,z)$ the interval partition of $\Sigma$ where $(k,i)$ and 
$(k,j)$ lie in the same block if and only if $k<n$ and $(\vv + z)_{k,i} = 
(\vv + z)_{k,j}$. Equivalently, an interval lies in $\II(\vv,z)$ if and only 
if it is an orbit of the action of $(S_\pi)_{z}$ on $\Sigma$. 
Let us say that $\sigma \in S_\pi$ is increasing, resp. decreasing, over an 
interval $\interval{a,b}_k \subset \Sigma$ if $\sigma_{(k)}(i) < 
\sigma_{(k)}(j)$, resp. $\sigma_{(k)}(i) > \sigma_{(k)}(j)$, whenever $a \leq 
i < j \leq b$. A permutation $\sigma$ is a $z$-shuffle if and only if it is 
increasing over every interval in $\II(\vv,z)$, so $\sigma^z$ is the unique 
permutation in $S_\pi$ which is increasing over all intervals in $\II(\vv,z)$ 
and such that $\sigma^z(z) = \sigma(z)$.

Given an interval $I = \interval{a,b}_k$ we denote by $\omega(I)$ the 
permutation $i \mapsto b+a-i$. This is the longest element in the symmetric
group of the interval $I$, seen as a subgroup of $S_\mu$. It follows that the 
longest element in $(S_\pi)_z$ is $\prod_{I\in\II(\vv,z)} \omega(I)$. We 
also write 
\begin{align*}
\alpha(I)
	&= (b \ b-1 \ldots \ a)_{(k)} 
	= (b \ b-1)_{(k)} (b-1 \ b-2)_{(k)} \cdots (a+1 \ a)_{(k)}, \\
\beta(I) 
	&=(a \ a+1 \ldots \ b)_{(k)}
	= (a \ a+1)_{(k)} (a+1 \ a+2)_{(k)} \cdots(b-1 \ b)_{(k)}. 
\end{align*}
These two permutations play a central role in the sequel. Notice that these 
permutations are mutual inverses.
\begin{Lemma}
\label{L:omega-delta}
Let $z\in \DD(\vv)$ and let $\omega_0$ be the longest element in $S_\pi$.
\begin{enumerate}[(a)]
\item 
\label{i:omega-z}
We have $\omega_0^z = \omega_0 \prod_{I \in \II(\vv,z)} \omega(I)$.

\item 
\label{i:D-delta}
Given $(k,i) \in \Sigma$, we have $z + \delta^{k,i} \in \DD(\vv)$ if and only 
	if $i = a(I)$ for some $I \in \II(\vv,z)[k]$, and $z-\delta^{k,i} 
	\in \DD(\vv)$ if and only if $i = b(I)$ for some $I \in \II(\vv,z)[k]$.

\item 
\label{i:omega-delta}
Let $I = \interval{a,b}_k \in \II(\vv,z)$ and set $u=z+\delta^{k,a}, 
	v = z- \delta^{k,b}$. There exist $\sigma, \tau \in S_\pi^z$ such that 
	$\omega_0^u = \sigma \alpha(I)$ and $\omega_0^v = \tau \beta(I)$. 
\end{enumerate}
\end{Lemma}
\begin{proof}
Put $\II' = \II(\vv,z)$. By definition $\omega_0$ is decreasing over each 
interval $I$ in $\II'$. 
Since $\omega(I)$ is decreasing over $I$, it follows that $\omega_0 \prod_{I 
\in \II'} \omega(I)$ is increasing over every interval $I \in \II'$, so it is 
a $z$-shuffle lying in the coclass $\omega_0 (S_\pi)_z$. This proves part 
(\ref{i:omega-z}). Part (\ref{i:D-delta}) follows from the 
definitions. 

To prove part (\ref{i:omega-delta}) we only have to show that 
$\sigma = \omega_0^u \beta(I)$ and $\tau = \omega_0^v \alpha(I)$ are 
$z$-shuffles. We verify this for $\sigma$, as the verification for $\tau$ is
similar. Let $J$ be the interval in $\II(\vv,u)$ corresponding to $(k,a)$, 
so $J = \interval{c,a}_k$ for some $c \leq a$. Taking $J' = J \setminus 
\{(k,a)\}$ and $I' = I \setminus \{(k,a)\}$, by part
(\ref{i:omega-z}),
\begin{align*}
\omega_0^u \beta(I)
	&= \omega_0 \left( \prod_{K \in \II(\vv, u), K \neq I',J} \omega(K) \right)
		\omega(J)\omega(I')\beta(I) \\
	&= \omega_0 \left( \prod_{K \in \II', K \neq I,J'} \omega(K) \right)
		\omega(J)\omega(I')\beta(I).
\end{align*}
It is enough to check that the composition of the product in the parenthesis 
with $\omega(J)\omega(I')\beta(I)$ is decreasing over the intervals of 
$\II'$. This is immediate for $K \neq I, J'$, so it remains to check that 
$\omega(J)\omega(I')\beta(I)$ is decreasing over $I, J'$. This follows 
immediately from the definitions
\end{proof}

\section{Background on universal tableaux Gelfand-Tsetlin modules}

Throughout this section we will work with the Lie algebra $\gl(n,\CC)$. We 
denote by $\h$ the Cartan subalgebra of diagonal matrices. We identify the 
dual of $\h$ with $\CC^n$ in the usual way.

\paragraph
\about{Generalities on Gelfand-Tsetlin modules} \label{subsec-gt-def}
For each $k \in \interval{n}$ we denote by $U_k$ the universal enveloping algebra of 
$\gl(k,\CC)$, and set $U = U_n$. By top-left-corner inclusion of matrices we 
obtain a chain
\begin{align*}
\gl(1,\CC) \subset \gl(2, \CC) \subset \cdots \subset \gl(n,\CC),
\end{align*}
which in turn induces a chain $U_1 \subset U_2 \subset \cdots \subset U_n$. 
Denote by $Z_k$ the center of $U_k$ and by $\Gamma$ the subalgebra of $U$ 
generated by $\bigcup_{k=1}^n Z_k$. This is a maximal commutative subalgebra 
of $U$ called the \emph{Gelfand-Tsetlin} subalgebra. It is generated by 
the elements
\begin{align*}
c_{k,i}
 &= \sum_{(r_1, \ldots, r_i) \in \interval{k}^i} 
 E_{r_1, r_2} E_{r_2, r_3} \cdots E_{r_i, r_1}
 & (k,i) \in \Sigma.
\end{align*}
By a result of Zhelobenko, there exists an isomorphism 
$$
\Gamma \to \CC[x_{k,i} 
\mid (k,i) \in \Sigma]^{S_\mu}
$$
given by $c_{k,i} \mapsto 
\gamma_{k,i}$, where
\begin{align*}
\gamma_{k,i}
 &= \sum_{j=1}^k (x_{k,j} + k - 1)^i 
 \prod_{m \neq j} \left( 
 1 - \frac{1}{x_{k,j} - x_{k,m}}
 \right),
\end{align*}
see \cite{FGR16}*{Subsection 3.1} for details. We will denote the image of 
$c \in \Gamma$ under this isomorphism by $\gamma_c$. It follows that $\Specm 
\Gamma \cong \CC^\mu / S_\mu$, hence, every $v \in \CC^\mu$ induces a 
character $\chi_v: \Gamma \to \CC$ by setting $c \mapsto \gamma_c(v)$. Notice 
that $\chi_v = \chi_w$ if and only if $w$ lies in the $S_\mu$-orbit of $v$. 

\begin{Definition}
A $U$-module $M$ is called a \emph{Gelfand-Tsetlin module} if it is finitely
generated and
\[
 M = \bigoplus_{\m \in \Specm \Gamma} M[\mathfrak m],
\] 
where $M[\mathfrak m] = \{x \in M \mid \mathfrak m^k x = 0 \mbox{ for some } k 
\geq 0\}$. The \emph{Gelfand-Tsetlin support}, or simply the \emph{support}, 
of $M$ is the set of all $\m$ such that $M[\m] \neq 0$, and will be denoted by 
$\supp M$. For every $\m \in \Specm \Gamma$ its \emph{Gelfand-Tsetlin 
multiplicity} in $M$ is $\dim M[\m]$.
\end{Definition}

Let $M$ be a Gelfand-Tsetlin module and let $v \in \CC^\mu$. We put
$M[v] = M[\ker \chi_v]$, and denote by $p_v: M \to M[v]$ the projection
map. We will identify the support of $M$ with the set of all $v \in \CC^\mu$ 
such that $M[v] \neq 0$. We will say that the elements of $M[v]$ have
Gelfand-Tsetlin weight $v$, and refer to $M[v]$ as \emph{the Gelfand-Tsetlin 
component of weight $v$}. We will usually say ''weight'' instead of 
''Gelfand-Tsetlin weight''. To avoid confusion we will sometimes use the 
expression ''Cartan weight‘‘ for elements in the dual of $\h$. Since $U(\h)
\subset \Gamma$ it follows that two elements with the same Gelfand-Tsetlin 
weight have the same Cartan weight, but the converse does not hold.

It is easy to check that a finitely-generated module $M$ is a Gelfand-Tsetlin 
module if and only if for each $x \in M$ the complex vector space $\Gamma x$ 
has finite dimension. The following lemma is an immediate consequence of that 
observation. 
\begin{Lemma}
\label{L:sub-gt}
Let $M$ be a Gelfand-Tsetlin module and let $N \subset M$ be a $U$-submodule.
Then $N$ is also a Gelfand-Tsetlin module. In particular for each $x \in N$
we have $p_v(x) \in N$ for all $v \in \CC^\mu$.
\end{Lemma}

\paragraph
\about{Universal tableaux Gelfand-Tsetlin modules}
\label{big-gt-modules}
Fix a seed $\vv$ and set $\pi = \pi(\vv), \II = \II(\vv)$ and $\DD = \DD(\vv)$.
In \cite{RZ18} a Gelfand-Tsetlin module $V(T(\vv))$ is associated to any 
seed $\vv$ (a similar construction appears in \cite{EMV18}). The module 
$V(T(\vv))$ was called the ``big Gelfand-Tsetlin module at $\vv$'' in 
\cite{RZ18}, but here we refer to it as the \emph{universal tableaux module 
associated to $\vv$}. It is a module with $\CC$-basis given by the set 
\begin{align*}
\{D_\sigma(\vv+z) \mid z \in \DD, \sigma \in S_\pi^z\}
\end{align*}
whose elements are called \emph{derivative tableaux}. A tableau of the form 
$D_e(\vv + z)$ is called the \emph{classical tableau} associated to $\vv + z$.
Given $z \in \DD(\vv)$ and $\sigma \in S_\pi^z$ we denote by $D_{<\sigma}(\vv 
+ z)$ an arbitrary linear combination of tableaux $D_\tau(\vv + z)$ with 
$\tau \in S_\pi^z$ strictly smaller than $\sigma$ in the induced Bruhat order
(for details on the Bruhat order on shuffles see \cite{BB05}*{Section 2.5}). 

We review the details regarding the explicit action of $U$ on $V(T(\vv))$, 
which were proved in \cite{FGRZ18}. Given $I = \interval{a,b}_k$ with $k <n$ 
we set
\begin{align*}
e_I
	&= \frac{\displaystyle \prod_{j = 1}^{k+1} (x_{k,a} - x_{k+1,j})}
		{\displaystyle \prod_{(k,j) \notin I} (x_{k,a} - x_{k,j})};
&f_I
	&= \frac{\displaystyle \prod_{j = 1}^{k-1} (x_{k,b} - x_{k-1,j})}
		{\displaystyle \prod_{(k,j) \notin I} (x_{k,b} - x_{k,j})}.
\end{align*}
Notice that if $I \in \II(\vv,z)$ then $e_I(\vv + z)$ and $f_I(\vv + z)$ are
well defined. We also set 
\begin{align*}
h_k = x_{k,1} + \cdots + x_{k,k} - (x_{k-1, 1} + \cdots + x_{k-1,k-1}) + k - 1.
\end{align*}

The following theorem is a direct consequence of \cite{FGRZ18}*{Lemma 8.4}.
\begin{Theorem}
\label{T:gt-big-module}
The action of the canonical generators of $\gl(n,\CC)$ on $V(T(\vv))$ is given 
by the formulas
\begin{align*}
E_{k,k+1} D_\sigma(\vv + z) 
	&= - \sum_{I \in \II(\vv,z)[k]} 
		\sum_{\tau \leq \sigma \alpha(I)} 
			\D_{\tau,\sigma\alpha(I)}^{\vv + z}(e_{I})
			D_{\tau}(\vv + z + \delta^{k,a(I)}), \\
E_{k+1,k} D_\sigma(\vv + z) 
	&= \sum_{I \in \II(\vv,z)[k]} 
		\sum_{\tau \leq \sigma \beta(I)} 
			\D_{\tau,\sigma\beta(I)}^{\vv + z}(f_{I})
			D_{\tau}(\vv + z - \delta^{k,b(I)}), \\
E_{k,k} D_{\sigma}(\vv + z) 
	&= h_k(\vv + z) D_{\sigma}(\vv+z),
\end{align*}
where $\D_{\tau,\sigma}$ are the Postnikov-Stanley operators introduced in
\cite{FGRZ18}*{Definition 3.1}, and elements $D_\tau(\vv + u)$
such that $\tau$ is not a $u$-shuffle should be treated as zero.
\end{Theorem}
Theorem \ref{T:gt-big-module} implies that every derivative tableau 
$D_\sigma(v)$ is a Cartan weight vector of weight $\lambda = (h_1(v), \cdots, 
h_n(v))$. Hence $V(T(\vv))$ is a (Cartan) weight representation with 
infinite-dimensional weight spaces. 
\begin{Remark}
\label{R:formulas}
We record here for future reference that $\D_{\tau,\sigma}$ is a differential 
operator of degree $\ell(\sigma) - \ell(\tau)$. In particular $\D_{\sigma,
\sigma}^v$ is the evaluation at $v$, which allows to rewrite the formulas in 
a simplified form as
\begin{align*}
E_{k,k+1} D_\sigma(\vv + z) 
	&= 
	 - \sum_{I \in \II(\vv,z)[k]} \bigg(
		e_I(\vv + z) D_{\sigma\alpha(I)}(\vv+z + \delta^{k,a(I)})\\
			&\qquad\qquad + D_{<\sigma\alpha(I)}(\vv+z + \delta^{k,a(I)})
			\bigg), \\
E_{k+1,k} D_\sigma(\vv + z) 
	&= 
	\sum_{I \in \II(\vv,z)[k]} \bigg(
	 	f_I(\vv + z) D_{\sigma\beta(I)}(\vv + z - \delta^{k,b(I)}) \\
	 		&\qquad\qquad
	 		+ D_{<\sigma\beta(I)}(\vv + z - \delta^{k,b(I)})
	 		\bigg),
\end{align*}
though we must keep in mind that in some cases $D_{\sigma\alpha(I)}(\vv+z+
\delta^{k,a(I)})$ and $D_{\sigma\beta(I)}(\vv+z - \delta^{k,b(I)})$ are zero.
This happens when $z_{k,a(I)} +1 = z_{k,a(I)-1}$ and $z_{k,b(I)} - 1 = 
z_{k,b(I) + 1}$ respectively. 
\end{Remark}

The following proposition shows that $V(T(\vv))$ is a Gelfand-Tsetlin module 
and describes the Gelfand-Tsetlin weight components of 
$V(T(\vv))$. For the proofs of the statements see \cite{FGRZ18}*{Proposition 
6.4 and Lemma 6.5}. 
\begin{Proposition}
\label{P:gt-weight-spaces}
Let $z \in \DD$, let $\m = \ker \chi_{\vv + z}$ and let $T = \sum_\sigma 
a_\sigma D_\sigma(\vv + z)$.
\begin{enumerate}[(a)]
\item 
\label{i:action}
If $c \in \Gamma$ then
\begin{align*}
c D_\sigma(\vv + z)
	&= \gamma_c(\vv+z)D_\sigma(\vv+z) +
		\sum_{\tau < \sigma} \D_{\tau,\sigma}^{\vv + z}(\gamma_c) 
			D_\tau(\vv+z)\\
	&= \gamma_c(\vv+z)D_\sigma(\vv+z) + D_{<\sigma}(\vv+z).
\end{align*}

\item 
\label{i:generator}
The Gelfand-Tsetlin component $V(T(\vv))[\m]$ is cyclic as a 
$\Gamma$-module, and $T$ is a cyclic vector if and only if $a_{\omega_0^z} 
\neq 0$.

\item 
\label{i:eigenvalue}
Let $r>0$. Then $\m^r T = 0$ if and only if $a_\sigma = 0$ for all 
$\ell(\sigma) \geq r$. In particular, $\m^{\ell(\omega_0^z) + 1} V(T(\vv))[\m] 
= 0$, and the only simultaneous $\Gamma$-eigenvector of eigenvalue 
$\chi_{\vv + z}$, up to a scalar multiple, is $D_e(\vv + z)$.

\item
\label{i:jordan} 
Let $r = \ell(\omega_0^z)$. Given $c \in \Gamma$ let $c_{\vv + z}$ be the 
restriction of $c$ on $V(T(\vv))[\m]$. Then the set of all $c$ such that 
the Jordan form of $c_{\vv + z}$ has exactly one block of size $r$ projects to 
a Zariski open set in $\Gamma/\m^{r+1}$. If $c_{\vv + z}$ falls 
outside this set then its Jordan form contains blocks only of size strictly 
smaller than $r$.
\end{enumerate}
\end{Proposition}
It follows from this proposition that $V(T(\vv))[\vv + z]$ is the $\CC$-span
of the derivative tableaux $D_\sigma(\vv + z)$. Notice that, through the map 
$D_\sigma (\vv + z) \mapsto \vv + \sigma(z)$, the support of $V(T(\vv))$ can 
be identified with $\vv + \ZZ^\mu_0$, and the Gelfand-Tsetlin multiplicity of 
$\vv + z$ is precisely the cardinality of its $S_\pi$-orbit. In 
\cite{FO14}*{Theorem 4.12 (c)} Futorny and Ovsienko proved that this is a 
bound for the dimension of the Gelfand-Tsetlin component of weight $\vv + z$ 
of a simple Gelfand-Tsetlin module and conjectured that the bound is sharp.
In particular these dimensions are bounded by $(n-1)!(n-2)! \cdots 2!$. As 
explained in the introduction, the conjecture can be proved using results
from \cite{FGRZ18} or \cite{EMV18}. We will refine this result in the following
sections.

\section{Cyclic submodules of universal tableaux modules}
Again, we fix a seed $\vv$ and set $\pi = \pi(\vv), \II = \II(\vv), \DD = 
\DD(\vv)$.

\paragraph
\about{The oriented graph of $\vv + z$}
\label{gt-chambers}
Recall that $\Omega(v)$ is the graph with vertex set $\Sigma$ and edges 
$[k,i] - [l,j]$ whenever $v_{k,i} - v_{l,j} \in \ZZ$ and $|k-l| \leq 1$; in
particular $\Omega(\vv) = \Omega(\vv + z)$ for all $z \in \ZZ^\mu_0$. In 
what follows we will define for each $z \in \DD$ an orientation of the 
graph $\Omega(\vv+z)$ and denote by $\overrightarrow \Omega (\vv+z)$ the 
resulting oriented graph. As usual, we will use the notation $[k,i] 
\rightarrow [l,j]$ for ``the oriented edge with tail $[k,i]$ and head 
$[l,j]$''. 

\begin{Definition}
\label{D:gt-chamber}
Let $z \in \DD$. The oriented graph $\overrightarrow \Omega (\vv+z)$ has 
$\Omega(\vv+z)$ as its underlying graph, and its orientation is subject to the 
following three rules.
\begin{itemize}
\item[(i)] 
If $[k,i] - [k,j]$ is an edge with $i < j$, then $[k,i] \rightarrow [k,j]$ is 
an edge of $\overrightarrow \Omega(\vv + z)$.

\item[(ii)] If $[k,i] - [k-1,j]$ is an edge such that $(\vv+z)_{k,i} - 
(\vv+z)_{k-1,j} \in \ZZ_{\geq 0}$, then $[k,i] \rightarrow [k-1,j]$ is an edge 
of $\overrightarrow \Omega(\vv + z)$.

\item[(iii)] If $[k,i] - [k-1,j]$ is an edge such that $(\vv + z)_{k,i} - 
(\vv + z)_{k-1,j} \in \ZZ_{<0}$, then $[k-1,j] \rightarrow [k,i]$ is an edge of
$\overrightarrow \Omega(\vv + z)$.
\end{itemize}
We denote by $\Omega^+(\vv+z)$ the subgraph of $\Omega(\vv+z)$ obtained by 
keeping only the edges of the form $[k,i] \rightarrow [k-1,j]$ in 
$\overrightarrow \Omega(\vv + z)$. Analogously we denote by $\Omega^-(\vv+ z)$ 
the subgraph obtained by keeping only the edges of the form $[k-1,j] 
\rightarrow [k,i]$ in $\overrightarrow \Omega(\vv + z)$. 
\end{Definition}
Notice that $\Omega^+(\vv + z)$ is an unoriented graph (though by its 
definition it is easy to recover the orientation of its edges in 
$\overrightarrow \Omega(\vv + z)$). Note also that the above definition can be 
considered as a refined version of the graph associated to a set of relations 
introduced in \S 4 of \cite{FRZ16a}. In the latter case, arrows between 
vertices on the $k$-th row are not allowed for $k < n$. 

It follows from the definition that $\overrightarrow \Omega(\vv + z)$ has no 
loops, so each of its connected components has at least one source (a vertex 
that is not the head of any edge) and at least one sink (a vertex that is not 
the tail of any edge). It also follows that $\Omega^-(\vv)$ is the graph
with vertex set $\Sigma$ and no edges, while for every $z \in \DD$ any edge 
of $\Omega^+(\vv)$ is an edge of either $\Omega^+(\vv + z)$ or of 
$\Omega^-(\vv + z)$.

Fix $z \in \DD$ and let $\overrightarrow \Omega = \overrightarrow 
\Omega(\vv+z)$. We next introduce a reduced version of the graph 
$\overrightarrow \Omega$ from which it can be recovered. We say that 
a directed edge $[k,i] \rightarrow [l,j]$ in $\overrightarrow \Omega$ is 
\emph{superfluous} if there exists a path of directed edges $[k,i] = [k_0, 
i_0] \rightarrow [k_1, i_1] \rightarrow \cdots \rightarrow [k_r, i_r] = [l,j]$ 
with $r > 1$. The \emph{reduced graph of $\vv+z$}, denoted by $\widetilde
\Omega(\vv+z)$, is the oriented subgraph obtained by removing all superfluous 
edges. 

Since $\overrightarrow\Omega(\vv+z)$ is a directed graph without loops, so is 
$\widetilde\Omega(\vv+z)$. We recover $\overrightarrow\Omega(\vv+z)$ from 
$\widetilde\Omega(\vv+z)$ by adding a directed edge $[k,i] \rightarrow 
[l,j]$ whenever there is a path from $[k,i]$ to $[l,j]$ in 
$\widetilde\Omega(\vv+z)$ and $|k-l| \leq 1$. Thus given $y \in \DD$ we have 
that $\overrightarrow \Omega(\vv + z) = \overrightarrow \Omega(\vv + y)$ if 
and only if $\widetilde \Omega(\vv + z) = \widetilde \Omega(\vv + y)$.

\begin{Example}
\label{e:graphs}
Below we show a few examples of reduced graphs $\widetilde \Omega(\vv + z)$. 

\vskip10pt

\begin{tabular}{|c|c|}
\hline
$\vv + z$ & $\widetilde \Omega(\vv+z)$
\\
\hline
\begin{tikzpicture}
%\node (v) at (-2.5,1.25) {$v = $};

\node (41) at (-1.5,2) {$0$};
\node (42) at (-0.5,2) {$0$};
\node (43) at (0.5,2) {$0$};
\node (44) at (1.5,2) {$0$};

\node (31) at (-1,1.5) {$0$};
\node (32) at (0,1.5) {$0$};
\node (33) at (1,1.5) {$0$};

\node (21) at (-.5,1) {$0$};
\node (22) at (.5,1) {$0$};

\node (11) at (0,0.5) {$0$};
\end{tikzpicture}
&
\tikzstyle{place}=[circle,draw=black,fill=black,
   inner sep=1pt,minimum size=1mm]
\begin{tikzpicture}
\node (41) at (-1.5,2) [place] {};
\node (42) at (-0.5,2) [place] {};
\node (43) at (0.5,2) [place] {};
\node (44) at (1.5,2) [place] {};

\node (31) at (-1,1.5) [place] {};
\node (32) at (0,1.5) [place] {};
\node (33) at (1,1.5) [place] {};

\node (21) at (-.5,1) [place] {};
\node (22) at (.5,1) [place] {};

\node (11) at (0,0.5) [place] {};

%\node (text) at (-2.5,1.25) {$\widetilde\Omega(v) = $};

\draw [->] (41) edge (42) (42) edge (43) (43) edge (44) (44) edge (31) 
	(31) edge (32) (32) edge (33) (33) edge (21) (21) edge (22) (22) edge 
	(11); 
\end{tikzpicture}\\
\hline
\begin{tikzpicture}
%\node (v) at (-2.5,1.25) {$v' = $};

\node (41) at (-1.5,2) {$3$};
\node (42) at (-0.5,2) {$2$};
\node (43) at (0.5,2) {$1$};
\node (44) at (1.5,2) {$0$};

\node (31) at (-1,1.5) {$2$};
\node (32) at (0,1.5) {$1$};
\node (33) at (1,1.5) {$0$};

\node (21) at (-.5,1) {$1$};
\node (22) at (.5,1) {$0$};

\node (11) at (0,0.5) {$0$};
\end{tikzpicture}
&
\tikzstyle{place}=[circle,draw=black,fill=black,
   inner sep=1pt,minimum size=1mm]
\begin{tikzpicture}
%\node (v) at (-2.5,1.25) {$\widetilde\Omega(v') = $};

\node (41) at (-1.5,2) [place] {};
\node (42) at (-0.5,2) [place] {};
\node (43) at (0.5,2) [place] {};
\node (44) at (1.5,2) [place] {};

\node (31) at (-1,1.5) [place] {};
\node (32) at (0,1.5) [place] {};
\node (33) at (1,1.5) [place] {};

\node (21) at (-.5,1) [place] {};
\node (22) at (.5,1) [place] {};

\node (11) at (0,0.5) [place] {};

\draw [->] 
	(41) edge (42) (42) edge (31) (31) edge (43) (43) edge (32) 
	(32) edge (44) (44) edge (33) (32) edge (21) (21) edge (33) (33) edge (22)
	(22) edge (11);
\end{tikzpicture}\\
\hline
\begin{tikzpicture}
\node (51) at (-2,2.5) {$a+1$};
\node (52) at (-1,2.5) {$a$};
\node (53) at (0,2.5) {$1$};
\node (54) at (1,2.5) {$0$};
\node (55) at (2,2.5) {$b$};

\node (41) at (-1.5,2) {$a+1$};
\node (42) at (-0.5,2) {$a$};
\node (43) at (0.5,2) {$b$};
\node (44) at (1.5,2) {$b-1$};

\node (31) at (-1,1.5) {$c+1$};
\node (32) at (0,1.5) {$c$};
\node (33) at (1,1.5) {$c$};

\node (21) at (-.5,1) {$a$};
\node (22) at (.5,1) {$a-1$};

\node (11) at (0,0.5) {$a+1$};
\end{tikzpicture}
&
\tikzstyle{place}=[circle,draw=black,fill=black,
   inner sep=1pt,minimum size=1mm]
\begin{tikzpicture}
\node (51) at (-2,2.5) [place] {};
\node (52) at (-1,2.5) [place] {};
\node (53) at (0,2.5) [place] {};
\node (54) at (1,2.5) [place] {};
\node (55) at (2,2.5) [place] {};

\node (41) at (-1.5,2) [place] {};
\node (42) at (-0.5,2) [place] {};
\node (43) at (0.5,2) [place] {};
\node (44) at (1.5,2) [place] {};

\node (31) at (-1,1.5) [place] {};
\node (32) at (0,1.5) [place] {};
\node (33) at (1,1.5) [place] {};

\node (21) at (-.5,1) [place] {};
\node (22) at (.5,1) [place] {};

\node (11) at (0,0.5) [place] {};

\draw [->] 
	(51) edge (41) (41) edge (52) (52) edge (42) (53) edge (54) 
	(55) edge (43) (43) edge (44) (31) edge (32) (32) edge (33) (11) edge (21)
	(21) edge (22);
\end{tikzpicture}
\\
\hline
\end{tabular}

\end{Example}
\medskip
\paragraph
\about{Cyclic submodules}
We now begin with our study of the internal structure of $V(T(\vv))$. We use 
the notation $\Omega^\pm(x) \subset \Omega^\pm(y)$ to indicate that the edge
set of $\Omega^\pm(x)$ is contained in that of $\Omega^\pm(y)$. We also denote 
by $\emptyset$ the graph with vertex set $\Sigma$ and no edges. The following 
lemma is a generalization of \cite{FGR15}*{Theorem 6.8}.

\begin{Lemma}
\label{L:omega-contained}
Let $y,z \in \DD$. If $\Omega^+(\vv+z) \subset \Omega^+(\vv+y)$, or 
equivalently if $\Omega^-(\vv + y) \subset \Omega^-(\vv+z)$, then the
following hold.
\begin{enumerate}[(a)]
\item 
\label{i:e}
$D_{e}(\vv + y) \in U D_{e}(\vv + z)$.

\item 
\label{i:omega}
$D_{\omega_0^y}(\vv + y) \in U D_{\omega_0^z}(\vv + z)$.
\end{enumerate}
\end{Lemma}
\begin{proof}
We apply induction on $|y-z| = \sum_{(k,i) \in \Sigma} |y_{k,i} - 
z_{k,i}|$. That is, we will show that it is possible to choose $(k,i) \in 
\Sigma$ and $u = z \pm \delta^{k,i}$ with the sign chosen so that $|y-u|
<|y-z|$ and the following hold:
\begin{enumerate}
\item $u \in \DD$,

\item $\Omega^+(\vv+z) \subset \Omega^+(\vv+u) \subset \Omega^+(\vv+y)$,

\item $D_{e}(\vv + u) \in U D_{e}(\vv + z)$, and

\item $D_{\omega_0^u}(\vv + u) \in U D_{\omega_0^z}(\vv + z)$.
\end{enumerate}
Clearly the lemma follows from the existence of such $u$.

Denote by $\Omega_<$, respectively $\Omega_>$, the induced subgraph of 
$\overrightarrow \Omega(\vv+y)$ with vertex set consisting of those $[k,i]$ 
such that $z_{k,i} < y_{k,i}$, respectively $z_{k,i} > y_{k,i}$; here induced 
means that there is a directed edge between two vertices of the subgraph if 
and only if there was a directed edge between these vertices in the original 
graph. Notice that no vertex of the form $[n,i]$ is in either graph. If both 
$\Omega_<$ and $\Omega_>$ are empty then $y = z$ and there is nothing to 
prove. Suppose $\Omega_<$ is not empty. Then, since it is an oriented subgraph 
of $\overrightarrow \Omega(\vv+y)$, it has no loops and 
hence has at least one source, say $[k,i]$. We claim that (1), (2), (3) and 
(4) hold with $u = z + \delta^{k,i}$. If $\Omega_<$ is empty then we take 
$[k,i]$ to be a sink in $\Omega_>$ and set $u = z - \delta^{k,i}$. We now 
proceed with the proof assuming $\Omega_<$ is not empty. The other case is
similar and we discuss how to adapt the proof at each step.

\emph{Proof of (1)}. By Lemma \ref{L:omega-delta}(\ref{i:D-delta}),
it is enough to show that if $[k,i-1] \to [k,i]$ is an edge of $\overrightarrow
\Omega(\vv)$ then $z_{k,i-1} > z_{k,i}$. If this edge is indeed present then 
since $y \in \DD$ we know that $y_{k,i-1} \geq y_{k,i}$. On the other hand 
since $[k,i]$ is a source of $\Omega_{<}$ we must have $z_{k,i-1} \geq 
y_{k,i-1}$ and $z_{k,i} < y_{k,i}$, and these inequalities imply the one we
are looking for.

\emph{Proof of (2).} To show that $\Omega^+(\vv+z) \subset \Omega^+(\vv+u)$ 
it is enough to prove that if either $[k,i] - [k-1,j]$ or $[k+1,j] - [k,i]$ 
is an edge of $\Omega^+(\vv+z)$ then it is also an edge of $\Omega^+(\vv+u)$. 
The first case is obvious. For the second, the choice of $(k,i)$ as a source 
of $\Omega_<$ implies that $z_{k+1,j} = u_{k+1,j} \geq y_{k+1,j}$ while 
$u_{k,i}
= z_{k,i} + 1 \leq y_{k,i}$. Since $\Omega^+(\vv+z) \subset \Omega^+(\vv+y)$
we see that $u_{k+1,j} \geq y_{k+1,j} \geq y_{k,i} \geq u_{k,i}$, so the edge
$[k+1,j] - [k,i]$ is in $\Omega^+(\vv+u)$.

To show that $\Omega^+(\vv+ u) \subset \Omega^+(\vv+ y)$ we again need to 
consider only edges of the form $[k+1,j] - [k,i]$ and $[k,i] - [k-1,j]$ 
of the first graph. In the first case we have $z_{k+1,j} = u_{k+1,j} \geq
u_{k,i} = z_{k,i} + 1$, so $[k+1,j] - [k,i]$ is an edge of $\Omega^+(\vv+z)$,
and by the hypothesis it is also and edge of $\Omega^+(\vv+y)$. In the second 
case we have $z_{k,i} + 1 = u_{k,i} \geq u_{k-1,j} = z_{k-1,j}$. If the 
inequality is strict then $[k,i] - [k-1,j]$ is an edge of $\Omega^+(\vv+z)$, 
and hence of $\Omega^+(\vv+y)$. If on the other hand equality holds, then 
$z_{k,i} < z_{k-1,j}$ and so $[k-1,j] \to [k,i]$ is an edge of $\overrightarrow
\Omega(\vv+z)$. Since $[k,i]$ is a source of $\Omega_<$, we must have that 
$y_{k-1,j} \leq z_{k-1,j} = z_{k,i} + 1 \leq y_{k,i}$, and hence $[k,i] - 
[k-1,j]$ is an edge of $\Omega^+(\vv + y)$.

\emph{Proof of (3).} Let $I \in \II(\vv,z)$ be the interval containing 
$(k,i)$. It follows from the definitions that $\alpha(I)$ is a $u$-shuffle.
Using the formulas for the action of $U$ as given in Remark \ref{R:formulas}, 
we see that
\begin{align*}
p_{\vv+u}&(E_{k,k+1} D_e(\vv + z)) =
	e_I(\vv + z) D_{\alpha(I)}(\vv + u) + D_{<\alpha(I)}(\vv + u),
\end{align*}
where $p_v$ is the projection to the Gelfand-Tsetlin component $V(T(\vv))[v]$. 
By Lemma \ref{L:sub-gt}, this element lies in $U D_e(\vv + z)$. 

If $e_I(\vv + z) = 0$, then there is some $(k+1,j) \in \Sigma$ such that 
$z_{k+1,j} = z_{k,i}$, which implies that $[k+1,j] - [k,i]$ is an edge of 
$\Omega^+(\vv+ z)$ and, hence, $y_{k+1,j} \geq y_{k,i}$. On the other hand
since $[k,i]$ is a source of $\Omega_<$ we have $z_{k,i} < y_{k,i} \leq 
y_{k+1,j} \leq z_{k+1,j}$, which leads to a contradiction. 
Thus we see that $p_{\vv+u}(E_{k,k+1} D_e(\vv + z)) \neq 0$. By Proposition 
\ref{P:gt-weight-spaces}(\ref{i:eigenvalue}) this implies that $D_e(\vv + u) 
\in U D_e(\vv + z)$. The proof in the second case is analogous, except that we 
must look at $p_{\vv+u}(E_{k+1,k} D_e(\vv + z))$, and instead of $e_I(\vv+z) 
\neq 0$, we show that $f_I(\vv + z) \neq 0$.

\emph{Proof of (4).} Recall from Lemma
\ref{L:omega-delta}(\ref{i:omega-delta}) that there exists $\sigma\in S_\pi^z$ 
such that $\omega_0^u = \sigma\alpha(I)$. By Proposition 
\ref{P:gt-weight-spaces} (\ref{i:generator}), $D_\sigma(\vv + z) \in 
U D_{\omega_0^z}(\vv + z)$. Again by the formulas in Remark \ref{R:formulas} 
and by Lemma \ref{L:sub-gt},
\begin{align*}
p_{\vv+u}&(E_{k,k+1} D_\sigma(\vv + z)) =\\
	&e_I(\vv + z) D_{\sigma\alpha(I)}(\vv + u) + D_{<\alpha(I)}(\vv + u)
	\in U D_{\omega_0^z}(\vv+z).
\end{align*}
As we saw in the proof of the previous point, the leading coefficient is 
nonzero.
By Proposition \ref{P:gt-weight-spaces}(\ref{i:generator}) this element 
generates the Gelfand-Tsetlin weight component $V(T(\vv))[\vv + u]$, so in 
particular $D_{\omega_0^u}(\vv + u) \in U D_{\omega_0^z}(\vv + z)$. For the 
second case we must take $p_{\vv+u}(E_{k,k+1} D_\tau(\vv + z))$, with $\tau$ 
as in Lemma \ref{L:omega-delta}(\ref{i:omega-delta}), and then the rest of the 
proof is similar.
\end{proof}

\paragraph
\about{The socle of $V(T(\vv))$}
An element $z \in \DD$ is said to be \emph{fully critical} if $z_{k,i} = 
z_{k,j}$ whenever $[k,i] - [k,j]$ is an edge of $\Omega(\vv + z)$. 
Equivalently, $z$ is fully critical if and only if $(S_\pi)_z = S_\pi$. The 
following proposition shows that $V(T(\vv))$ is a cyclic module, and that it 
has a unique minimal module. Item \ref{i:simple} of the next Proposition is a 
special case of \cite{EMV18}*{Theorem 11} and \cite{FGRZ18}*{Theorem 8.5}. We
include it here for completeness.

\begin{Proposition}
\label{P:cyclic}
Let $z \in \DD$ be fully critical.
\begin{enumerate}[(a)]
\item 
\label{i:cyclic}
If $\Omega^+(\vv+ z) = \emptyset$ then $U D_e(\vv + z) = V(T(\vv))$. In 
particular $V(T(\vv))$ is a cyclic module.

\item
\label{i:minimal-module}
If $\Omega^-(\vv+z) = \emptyset$ then $U D_e(\vv + z)$ is simple and contained
in any other submodule of $V(T(\vv))$.

\item
\label{i:simple}
If $\Omega(\vv)$ has no edges of the form $[k,i] - [k-1,j]$ then $V(T(\vv))$ 
is simple.
\end{enumerate}
\end{Proposition}
\begin{proof}
Notice that it is always possible to find fully critical elements $z \in 
\DD$ satisfying the hypothesis of the first item. If $z$ is such an 
element, being fully critical, then $\omega_0^z = e$. Also by Lemma 
\ref{L:omega-contained} (\ref{i:omega}), every derivative tableau 
$D_{\omega_0^y}(\vv + y)$ with $y \in \DD$ is in $U D_e(\vv+z)$. Thus 
\[
U D_e(\vv+z)[\vv + y] 
	\supset \Gamma D_{\omega_0^y}(\vv+y) 
	= V(T(\vv))[\vv+y]
\]
and this proves part (\ref{i:cyclic}).

Assume now that $\Omega^-(\vv+z) = \emptyset$, and set $N = U D_e(\vv + z)$. 
Then for any $y \in \DD$ we have $\Omega^+(\vv+y)\subset\Omega^+(\vv+z)$, 
so by Lemma \ref{L:omega-contained}(\ref{i:e}) $N \subset U D_e(\vv + y)$.
Since every submodule of $V(T(\vv))$ contains some tableaux of this form, it
follows that $N$ is contained in every submodule of $V(T(\vv))$ which 
proves part (\ref{i:minimal-module}). Part (\ref{i:simple}) is an easy 
consequence of (\ref{i:cyclic}) and (\ref{i:minimal-module}).
\end{proof}

As mentioned in the proof of the Proposition, if $\Omega^-(\vv+z) = \emptyset$
then $U D_e(\vv + z)$ is minimal among the submodules of $V(T(\vv))$. 
\begin{Corollary}
The socle of $V(T(\vv))$ is simple and equal to $U D_e(\vv + z)$ for any $z$
such that $\Omega^-(\vv+z) = \emptyset$. In particular, 
$\soc~V(T(\vv)) = U D_e(\vv)$. 
\end{Corollary}

\section{The essential support of the socle and a proof of the Strong
Futorny-Ovsienko conjecture}
As before, $\vv$ is a fixed seed and $\pi = \pi(\vv), \II = \II(\vv), \DD = 
\DD(\vv)$.

\paragraph
\about{The essential support of the socle of $V(T(\vv))$} 
\label{subsec-socle}
We denote the socle of $V(T(\vv))$ by $V_\soc$. As mentioned above $V_\soc$ is 
simple. The dimension of the Gelfand-Tsetlin component $V_\soc[\vv+z]$ is 
bounded by $\dim V[\vv+z] = |S_\pi^z|$ (this is the same 
bound given in \cite{FO14}*{Theorem 4.12 (c)}). 
\begin{Definition}
Let $z \in \DD$ and let $M$ be a simple Gelfand-Tsetlin module such that 
$M[\vv + z] \neq 0$. We say that $\vv + z$ is in the \emph{essential support} 
of $M$ if $\dim M[\vv+z] = |S_\pi^z|$. We denote the essential support of $M$
by $\essupp M$.
\end{Definition}
We introduce the following notations
\begin{align*}
\supp(\vv)
	&= \{z \in \DD \mid \dim V_\soc[\vv+z] \neq 0\} \\
\ess(\vv)
	&= \{z \in \DD \mid \vv + z \in \essupp V_\soc \}
\end{align*}
In view of the fact that $V_\soc$ is simple, its essential support is the set 
of those $z \in \DD$ for which the dimension of $V_\soc[\vv+z]$ is as large as 
the Futorny-Ovsienko bound allows. 

\paragraph
\about{Cones associated to the support}
Denote by $\RR^\mu$ the set of points in $\CC^\mu$ with real coordinates. 
Recall that a rational polyhedral cone is the intersection of finitely many
half-spaces $\{x \in \RR^\mu \mid \phi(x) \geq q\}$ where $\phi$ is a linear 
functional with rational coefficients in the canonical basis and $q \in \QQ$. 
The \emph{rank} of a cone is the dimension of the smallest affine space that
contains it.

Set 
\begin{align*}
	\mathcal S &= \{\Omega^+(\vv+z) \mid z \in \supp(\vv)\} \\
	\mathcal E &= \{\Omega^+(\vv+z) \mid z \in \ess(\vv)\}. 
\end{align*} 
These are finite sets, which we order by the relation of being a subgraph. 
For each $\Omega \subset \Omega^+(\vv)$ we set
\begin{align*}
\P(\Omega) 
	&= \{z \in \DD \mid \Omega = \Omega^+(\vv+z)\};\\
\overline \P(\Omega) 
	&= \{z \in \DD \mid \Omega \subset \Omega^+(\vv+z)\}.
\end{align*}
By Lemma \ref{L:omega-contained} if $z \in \supp(\vv)$ then 
$\P(\Omega^+(\vv + z)) \subset \overline \P(\Omega^+(\vv + z)) \subset 
\supp(\vv)$, and the same holds if we replace $\supp(\vv)$ with $\ess(\vv)$. 
Denoting by $\mathcal S_0$ and $\mathcal E_0$ the set of minimal elements of 
$\mathcal S$ and $\mathcal E$ respectively we see that
\begin{align*}
\supp(\vv) 
	&= \bigcup_{\Omega \in \mathcal S} \P(\Omega)
	= \bigcup_{\Omega \in \mathcal S_0} \overline \P(\Omega); \\
\ess(\vv) 
	&= \bigcup_{\Omega \in \mathcal E} \P(\Omega)
	= \bigcup_{\Omega \in \mathcal E_0} \overline \P(\Omega).
\end{align*}

Let $w\in \DD$ and $z \in \ZZ^\mu_0$, and put $\overrightarrow \Omega = 
\overrightarrow \Omega(\vv + w)$. We have that $z \in \P(\Omega^+(\vv + w))$ 
if and only if it satisfies the following conditions for all $k < n$:
\begin{enumerate}[$(a)$]
\item if $[n,i] \to [n-1,j]$ is a directed edge of $\overrightarrow \Omega$
then $z_{n-1,j} \leq \vv_{n,i} - \vv_{n-1,j}$;

\item if $[n-1,j] \to [n,i]$ is a directed edge of $\overrightarrow \Omega$
then $z_{n-1,j} > \vv_{n,i} - \vv_{n-1,j}$;

\item if $[k,i] \to [k-1,j]$ is a directed edge of 
$\overrightarrow \Omega$ then $z_{k-1,j} \leq z_{k,j}$ ;

\item if $[k-1,i] \to [k,j]$ is a directed edge of 
$\overrightarrow \Omega$ then $z_{k-1,j} > z_{k,j}$;

\item if $[k,i] \to [k,i+1]$ is a directed edge of 
$\overrightarrow \Omega$ then $z_{k,i} \geq z_{k,i+1}$.
\end{enumerate}
Furthermore, $z \in \overline \P(\Omega^+(\vv + w))$ if and only if it 
satisfies conditions $(a), (c)$ and $(e)$ (this last condition guarantees that 
$z \in \DD$). It follows that both $\P(\Omega^+(\vv + w))$ and $\overline 
\P(\Omega^+(\vv + w))$ are the set of integral points of a rational polyhedral 
cone. Since these sets are contained in $\ZZ^\mu_0$ their rank is at most 
$\frac{n(n-1)}{2}$.

\begin{Theorem}
\label{T:essupp}
Both the support and the essential support of $V_\soc$ are the set of integral 
points of a finite union of polyhedral rational cones. Furthermore, one of 
these cones has rank $\frac{n(n-1)}{2}$.
\end{Theorem}
\begin{proof}
We have already shown that both $\supp V_\soc$ and $\essupp V_\soc$ can be
written as a union of sets of the form $\overline \P(\Omega)$ for appropriate
subgraphs of $\Omega(\vv)$, so the first part of the theorem is proved.

For the second part, we will show show that $\P(\Omega^+(\vv)) = \overline 
\P(\Omega^+(\vv))$ is the set of integral points of a cone of rank 
$\frac{n(n-1)}{2}$. Given $1 \leq i \leq k \leq n-1$ let $c^{k,i} \in 
\ZZ^\mu_0$ be such that $c^{k,i}_{l,j}$ is $-1$ if there is a directed path 
in $\overrightarrow 
\Omega(\vv)$ from $[k,i]$ to $[l,j]$, and $0$ otherwise. Here we assume that 
$[k,i]$ is linked to itself by a trivial path, so $c^{k,i}_{k,i} = -1$. Then 
$\Omega^+(\vv+c^{k,i})=\Omega^+(\vv)$ and so $c^{k,i} \in \P(\Omega^+(\vv))$. 
Now if $z$ is any nonzero element in $\P(\Omega^+(\vv))$ and $z_{k,i} \neq 0$ 
then $\Omega^+(\vv + z - c^{k,i}) = \Omega^+(\vv)$. By induction 
$\P(\Omega^+(\vv))$ is the monoid generated by the $c^{k,i}$, and hence has 
the desired rank.
\end{proof}

Now we are ready to prove the Strong Futorny-Ovsienko Conjecture. For each 
standard parabolic subgroup $G \subset S_\pi$ there exists $z$ in 
$\P(\Omega(\vv))$ such that $G = (S_\pi)_z$. Indeed, the orbits of $G$ form
an interval partition of $\Sigma'$, and the desired element is obtained by 
taking the inequalities $(e)$ presented above as equalities whenever $(k,i), 
(k,i+1)$ lie in the same $G$-orbit, and as strict inequalities if they lie in 
different orbits. Thus we have the following result.

\begin{Corollary}
\label{cor-fo-conj}
Let $z \in \DD$ and let $G = (S_\pi)_z$. Then there exists $z' \in \ess(\vv)$ 
such that $(S_\pi)_{z'} = G$. Thus $\dim V_\soc[\vv+z'] = \frac{|S_\pi|}{|G|}$ 
and the FO inequality is sharp for all parabolic subgroups of $S_\pi$.
\end{Corollary}
In the next section we will show that in the case were $\vv = \mathbf 0$ this 
bound can be obtained when $M$ a simple Verma module. 

We conclude this section with the observation that Corollary 
\ref{cor-fo-conj} implies that if the essential support of a simple 
Gelfand-Tsetlin module intersects $\ess(\vv)$, then it must be isomorphic to
$\soc~V(T(\vv))$.
\begin{Corollary}
If $z \in \ess(\vv)$ then the module $V_{\rm soc} = U D_e(\vv)$ is, up to
isomorphism, the unique simple Gelfand-Tsetlin module having $\vv +z$ 
in its support.
\end{Corollary}
\begin{proof}
Let $\m$ be the maximal ideal of $\Gamma$ corresponding to $\vv$ and let $N = 
U/U\m$. Then being a simple module, $V_\soc$ appears as a quotient of $N$. 
By \cite{FO14}*{Theorem 4.12(c)}, $\dim N[\vv +z] \leq \frac{|S_\pi|}
{|(S_{\pi})_z)|}$, while $\dim V_\soc[\vv+z] = \frac{|S_\pi|}{|(S_{\pi})_z)|}$ 
for all $z \in \ess(\vv)$. This implies the corollary.
\end{proof}

\section{Realization of Verma modules in universal tableaux modules}
Let $\lambda \in \CC^n$ be a (Cartan) weight of $\gl(n,\CC)$. Recall that we 
can associate to $\lambda$ a Verma module $M(\lambda)$ by extending $\lambda$ 
to a character of the Borel subalgebra of upper-triangular matrices and then 
inducing the resulting module to $\gl(n,\CC)$. The module $M(\lambda)$ 
is a highest weight module with highest weight $\lambda$ and any other such 
module is a quotient of $M(\lambda)$.

\paragraph
\about{Restriction from $\gl(n,\CC)$ to $\sl(n,\CC)$}
We fix $\h' \subset \sl(n,\CC)$ to be the usual Cartan subalgebra of diagonal
matrices of trace $0$.

Given a representation $V$ of $\gl(n,\CC)$ we denote by $V'$ its restriction 
to $\sl(n,\CC)$, and given $v \in V$ we denote by $v'$ the corresponding vector
in $V'$. Then $v$ is a Cartan weight vector if and only if $v'$ is a 
Cartan weight vector, and if $\lambda = (\lambda_1, \ldots, \lambda_n)$ is the 
weight of $v$ then the weight of $v'$ is $\lambda' = (\lambda_1 - \lambda_2, 
\ldots, \lambda_{n-1} - \lambda_n)$. 
Verma modules restrict to Verma modules, and $M(\lambda)' = M(\lambda')$. In
particular this implies that $M(\lambda)$ is a free $U^-$-module, where
$U^- \subset U$ is the subalgebra generated by the elements $E_{j,i}$ with
$1 \leq i < j \leq n$. We will say that $\lambda$ is dominant, resp. integral, 
resp. dominant integral, if the corresponding $\sl(n,\CC)$-weight is dominant, 
resp. integral, resp. dominant integral.

The natural action of $S_n$ on $\CC^n$ induces an action of $S_n$ on $\CC^n/L$,
where $L$ is the vector space generated by the vector $(1,1, \ldots, 1)$. If
we identify $\CC^n/L$ with the dual of the Cartan subalgebra of $\sl(n,\CC)$
and the corresponding Weyl group with $S_n$, then the induced action of $S_n$ 
and the action of the Weyl group coincide. In other words, for each $\sigma 
\in S_n$ and each $\gl(n,\CC)$-weight $\lambda$ we have $\sigma(\lambda)' = 
\sigma(\lambda')$.

Let $\tilde \rho = -(0,1,\ldots, n-1)$. Then $\tilde \rho'$ is the half-sum of 
the positive roots of $\sl(n,\CC)$. The dot action of $S_n$ on $\CC^n$ is 
given by $\sigma \cdot \lambda = \sigma(\lambda + \tilde \rho) - \tilde \rho$ 
for all $\sigma \in S_n$. This dot action induces the dot action of $S_n$ on 
$\sl(n,\CC)$-weights, i.e. $(\sigma\cdot\lambda)' = \sigma \cdot \lambda'$.

\paragraph
\about{Highest weight vectors of universal tableaux modules}
For each $\lambda \in \CC^n$ we set $\tilde \lambda = \lambda + \tilde \rho$.

Let $\lambda \in \CC^n$ be such that $\tilde \lambda$ is a dominant 
weight, and let $\sigma \in S_n$. We denote by $HW(\lambda, \sigma)$ an 
element of $\CC^\mu$ in normal form whose $k$-th row is a permutation of 
$(\tilde \lambda_{\sigma^{-1}(1)}, \tilde \lambda_{\sigma^{-1}(2)}, \ldots, 
\tilde \lambda_{\sigma^{-1}(k)})$ for each $k$; this element may not be 
unique, so we fix one for each choice of $\lambda$ and $\sigma$. If $\tilde 
\lambda$ is dominant integral then the top row of $HW(\lambda,\sigma)$ is 
equal to $\tilde \lambda$. 

Let $\tau \in S_n$. If $\sigma(\lambda)-\tau(\lambda) \in \ZZ^n$ then 
we can and will assume that $HW(\lambda,\sigma)$ and $HW(\lambda,\tau)$ lie in 
the same $(\ZZ^\mu_0 \# S_\mu)$-orbit. 

\begin{Lemma}
\label{L:weight-tableau}
Let $\vv$ be a seed in the $(\ZZ^\mu_0\# S_\mu)$-orbit of $HW(\lambda,
\sigma)$. Then the classical tableau $D_e(HW(\lambda, \sigma)) \in V(T(\vv))$ 
is a highest weight vector of $V(T(\vv))$ of weight $\sigma \cdot \lambda = 
\sigma(\lambda + \tilde \rho) - \tilde \rho$.
\end{Lemma}
\begin{proof}
Put $v = HW(\lambda, \sigma)$ and let $k \in \interval{n}$. The fact that 
$D_e(v)$ has the desired weight follows from the 
definition of $v$ and the formula for the action of $E_{k,k}$. It remains to 
prove that $D_e(v)$ is a highest weight vector. 

Theorem \ref{T:gt-big-module} implies that $E_{k,k+1} D_e(v)$ is a linear 
combination of derivative tableaux $D_\tau(v + \delta^{k,i})$ where $(k,i)$ is 
the first entry in an interval $I = \{(k,i), (k,i+1), \ldots, (k,j)\}$ and 
$\tau < \alpha(I)$. Furthermore, the coefficient of $D_\tau(v + \delta^{k,i})$ 
is $\D_{\tau,\alpha(I)}(e_I)(v)$. To prove $E_{k,k+1} D_e(v) =0$ it is enough 
to show that $\D_{\tau,\alpha(I)}(e_I)(v) = 0$ for all such $\tau$.

By definition, the interval $I$ is such that $v_{k,i} = v_{k,i+1} = \ldots =
v_{k,j}$, and since the $k$-th row of $v$ is obtained by deleting one element 
from its $k+1$-th row, there are at least $|I|$ entries in the $k+1$-th row of 
$v$ equal to $v_{k,i}$. Thus $e_I$ is a rational function with a zero of order
at least $|I|$ in $v$. On the other hand, by definition $\D_{\tau,\alpha(I)}$ 
is a differential operator of order $\ell(\alpha(I)) - \ell(\tau) < |I|$, and 
hence $\D_{\tau,\alpha(I)}(e_I)(v) = 0$. 
\end{proof}

We fix a seed $\vv$ in the $(\ZZ^\mu_0\# S_\mu)$-orbit of $HW(\lambda,\sigma)$ 
(again, the top row of $\vv$ equals $\tilde \lambda$ if this is dominant 
integral). The lemma and the universal property of Verma modules together 
imply that there is a map $j_{\lambda,\sigma}: M(\sigma\cdot\lambda)\to 
V(T(\vv))$ mapping a highest weight vector $v_{\sigma\cdot\lambda}$ to 
$D_e(HW(\lambda, \sigma))$ for each $\sigma \in S_n$. As the following theorem 
shows, this map is injective and, therefore, any Verma module can be realized 
as a submodule of an adequate universal tableaux module.

\begin{Theorem}
\label{T:verma-injection}
For each dominant weight $\lambda$ and each $\sigma \in S_n$ the map 
$j_{\lambda, \sigma}$ is injective. In particular every Verma module 
can be realized as a submodule of a universal tableaux module. 
\end{Theorem}
\begin{proof}
By the classical Verma and BGG theorems, the Verma module $M(\sigma \cdot 
\lambda)'$ has a simple socle, which is isomorphic to $M(\nu')$, where $\nu$ 
is the unique element in the orbit of $S_n \cdot \lambda$ such that $\nu'$ is 
both linked to $\sigma \cdot \lambda'$ and antidominant (see for example 
\cite{Humphreys08}*{Chapter 4}). These conditions imply that $v = HW(\lambda,
\sigma)$ and $w = HW(\lambda, \nu)$ lie in the same $\ZZ^\mu_0 \# S_\mu$-orbit.
In particular, they are both highest weight vectors in $V(T(\vv))$ by Lemma
\ref{L:weight-tableau}.

Since $M(\nu)$ is simple and $j_{\lambda,\nu}$ is nonzero, it defines an
isomorphism $M(\nu) \cong U D_e(w)$, and since the socle of $V(T(\vv))$ is 
simple, it must be equal to $U D_e(w)$. By Proposition \ref{P:cyclic} 
(\ref{i:minimal-module}) the socle of $V(T(\vv))$ is contained in $U D_e(v)$, 
so the image of $j_{\lambda, \sigma}$ is a quotient of $M(\sigma\cdot\lambda)$ 
containing a highest weight vector of weight $\nu \cdot \lambda$. By 
\cite{Humphreys08}*{Theorem 7.16} this is possible only if the map 
$j_{\lambda,\sigma}$ is injective.
\end{proof}

\paragraph
\about{The support of the Verma module $M(-\tilde\rho)$}
It follows from Lemma \ref{L:weight-tableau} that the tableau $D_e(\mathbf 0)$ 
is a highest weight vector of weight $-\tilde\rho$ in $V(T(\mathbf 0))$. Since 
the Verma module $M(-\tilde\rho)$ is simple, by Theorem \ref{T:verma-injection}
it is isomorphic to the socle of the corresponding universal tableaux module. 
This allows us to study both its Gelfand-Tsetlin support and its essential 
support. We begin our study of the support of $M(-\tilde\rho)$ with the 
following lemma.

\begin{Lemma}
\label{L:zero-nonzero}
Let $z \in \DD(\mathbf 0)$, let $I=\interval{a,b}_k\in \II(\mathbf 0,z)$, let 
$m = z_{k,a} = z_{k,a+1} = \cdots = z_{k,b}$, and put $z'=z - \delta^{k,b}$. 
Also, let $J$ be the (possibly empty) interval formed by the $(k-1,j) \in 
\Sigma$ such that $z_{k-1,j} = m$. Then $p_{z'}(E_{k+1,k} D_e(z))$ is nonzero 
if and only if $|J| \leq |I|$.
\end{Lemma}
\begin{proof}
Recall that $\beta = \beta(I)$ is the permutation $(a \ a+1 \ \cdots
b)_{(k)}$. The formulas for the action of $U$ on $V(T(\mathbf 0))$ given in 
Theorem \ref{T:gt-big-module} tell us that 
\begin{align*}
p_{z'}(E_{k+1,k} D_e(z)) 
	&= \sum_{\tau \leq \beta} \D_{\tau,\beta}(f_I)(z) D_\tau(z').
\end{align*}
The differential operator $\D_{\tau,\beta}$ has order $\ell(\beta) - \ell(\tau)
\leq \ell(\beta) = |I|$, and by definition $f_I$ has a zero of order $|J|$ at 
$z$. Thus $\D_{\tau,\beta(I)}(f_I)(z) = 0$ if $|J| > |I|$.

Now suppose $|J| \leq |I|$ and let $\sigma = (a+|J| \ a+|J|+1 \ \cdots \ 
b)_{(k)}$. Then the order of $\D = \D_{\sigma,\beta}$ is $\ell(\beta) - 
\ell(\sigma) = |J|$. If we denote by $\partial_{k,i}$ the partial derivative 
with respect to $x_{k,i}$ then $\D$ is a homogeneous polynomial of degree 
$|J|$ in $\partial_{k,i}$ with $a \leq i \leq b$. Using the definition of 
Stanley-Postnikov differential operators from \cite{FGRZ18}*{paragraph 3.3}, 
we see that the coefficient of $\partial_{k,b}^{|J|}$ in $\D$ is precisely 
$\frac{(-1)^{|J|}}{|J|!}$. On the other hand, since $f_I$ is a rational 
function on the variables $x_{k,b}$ and $x_{k-1,j}$ with $(k-1,j) \in J$, we 
see that $\D(f_I) = \frac{(-1)^{|J|}}{|J|!}\partial_{k,b}^{|J|}(f_I)$. Thus
\begin{align*}
\D(f_I)(z)
	&= (-1)^{|J|}\frac{\displaystyle \prod_{(k-1,j) \notin J} m - z_{k-1,j}}
		{\displaystyle \prod_{(k,j) \notin I} z_{k,a} - z_{k,j}} \neq 0. 
\end{align*} 
Since this is the coefficient of $D_\sigma(z')$ in $E_{k+1,k} D_e(z)$, we are 
done.
\end{proof}

From this point on we identify $M(-\tilde\rho)$ with $U D_e(\mathbf 0)$. We 
have already presented the graph $\widetilde \Omega(\mathbf 0)$ for $n = 4$ in 
Example \ref{e:graphs}. In general, its edges are those of the form $[k,i] \to 
[k,i+1]$ for $1 \leq i \leq k-1$ and $[k,k] \to [k-1,1]$ for all $2 \leq k 
\leq n$. 

Let $v \in \ZZ^\mu_0$ be the element with $v_{k,i} = -i + 1$ for all $1 \leq i 
\leq k \leq n-1$. For example, for $n = 5$ we have

\medskip
\begin{tabular}{|c|c|}
\hline
$v$ & $\widetilde \Omega(v)$
\\
\hline
\begin{tikzpicture}
%\node (v) at (-2.5,1.25) {$v' = $};

\node (51) at (-2,2.5) {$0$};
\node (52) at (-1,2.5) {$0$};
\node (53) at (0,2.5) {$0$};
\node (54) at (1,2.5) {$0$};
\node (55) at (2,2.5) {$0$};

\node (41) at (-1.5,2) {$0$};
\node (42) at (-0.5,2) {$-1$};
\node (43) at (0.5,2) {$-2$};
\node (44) at (1.5,2) {$-3$};

\node (31) at (-1,1.5) {$0$};
\node (32) at (0,1.5) {$-1$};
\node (33) at (1,1.5) {$-2$};

\node (21) at (-.5,1) {$0$};
\node (22) at (.5,1) {$-1$};

\node (11) at (0,0.5) {$0$};
\end{tikzpicture}
&
\tikzstyle{place}=[circle,draw=black,fill=black,
   inner sep=1pt,minimum size=1mm]
\begin{tikzpicture}
%\node (v) at (-2.5,1.25) {$\widetilde\Omega(v') = $};
\node (51) at (-2,2.5) [place] {};
\node (52) at (-1,2.5) [place] {};
\node (53) at (0,2.5) [place] {};
\node (54) at (1,2.5) [place] {};
\node (55) at (2,2.5) [place] {};

\node (41) at (-1.5,2) [place] {};
\node (42) at (-0.5,2) [place] {};
\node (43) at (0.5,2) [place] {};
\node (44) at (1.5,2) [place] {};

\node (31) at (-1,1.5) [place] {};
\node (32) at (0,1.5) [place] {};
\node (33) at (1,1.5) [place] {};

\node (21) at (-.5,1) [place] {};
\node (22) at (.5,1) [place] {};

\node (11) at (0,0.5) [place] {};

\draw [->] 
	(51) edge (52) (52) edge (53) (53) edge (54) (54) edge (55) 
	(55) edge (41) (41) edge (31) (31) edge (42) (42) edge (32) (32) edge (43) 
	(43) edge (33) (33) edge (44)
	(31) edge (21) (21) edge (32) (32) edge (22) (22) edge (33)
	(21) edge (11) (11) edge (22) ;
\end{tikzpicture}\\
\hline
\end{tabular}
\\
\medskip

The edges of $\Omega^+(v)$ are precisely those of the form $[n,i] - [n-1,j]$ 
for all $i \in \interval{n}, j \in \interval{n-1}$, along with all $[k,i] - 
[k-1,j]$ for $1 \leq i \leq j \leq k \leq n - 1$. Recall from 
\S\ref{subsec-socle} that $\overline \P (\Omega^+(v))$ is the set of those $z 
\in \DD(\mathbf 0)$ such that $\Omega^+(v) \subset \Omega^+(z)$. We will prove 
that the support of $M(-\rho)$ is equal to the set of integral points 
$\overline \P (\Omega^+(v))$.

\begin{Theorem}\label{thm-supp Verma}
Let $z \in \DD(\mathbf 0)$. Then $z$ lies in the Gelfand-Tsetlin support of 
$M(-\rho)$ if and only if $z_{k-1,j} \leq z_{k,i} \leq 0$ for all $k \leq n-1$ 
and $i \leq j$. Equivalently, the support of $M(-\rho)$ coincides with 
$\overline \P(\Omega^+(v))$.
\end{Theorem}
\begin{proof}
We first show that $\overline \P(\Omega^+(v))$ is contained in the support
of $M(-\tilde\rho)$. By Lemma \ref{L:omega-contained} it is enough to show that
$D_e(v) \in M(-\tilde\rho)$. 

For each $l \in \interval{n-1}$, let $v(l)$ be the element of $\ZZ^\mu_0$ such 
that $v(l)_{k,i} = -i + 1$ if $k \leq l$, while $v(l)_k = 0$ if $k > l$; in 
other words, the $k$-th row of $v(l)$ is equal to that of $v$ if $k \leq l$ 
and equal to zero if $k > l$, in particular $v = v(n-1)$. We now show that 
$v(l)$ lies in the support of $M(-\tilde\rho)$ by induction on $l$, the case 
$v(1) = \mathbf 0$ being obvious. Suppose $v(l-1)$ lies in the support. Then 
row $l$ has exactly $l$ entries equal to $0$, while row $l-1$ has exactly one 
entry that equals $0$. If we apply Lemma \ref{L:zero-nonzero} $l-1$ times to 
rows $l$ and $l-1$ with $m = 0$ we obtain that $v^{(1)} = v(l-1) - 
(\delta^{l,2} + \delta^{l,3} + \cdots + \delta^{l,l})$ lies in the support of 
$M(-\tilde\rho)$. Applying now the same lemma $l-2$ times to the same rows for
$m = -1$, we see that $v^{(2)} = v^{(1)} - (\delta^{l,3} + \cdots + 
\delta^{(l,l)})$ also lies in the support of $M(-\tilde\rho)$. Repeating this 
argument for $m = -2, -3, \cdots, -l+1$, we see that 
\[
v^{(l-1)} = v(l-1) - (\delta^{2,l} + 2 \delta^{3,l} + \cdots + (l-1) 
\delta^{l,l})
\] 
is in the support of $M(-\tilde\rho)$. Thus $v = v(n-1)$ lies in the support.

We now prove that the support of $M(-\tilde \rho)$ is contained in 
$\P(\Omega^+(v))$. Suppose there exists $z$ in the support of $M(-\tilde\rho)$ 
outside of $\overline \P(\Omega^+(v))$, that is, that there exists 
some $j \leq i$ such that $z_{k,i} > z_{k+1,j}$. Since $z \in \DD(\mathbf 0)$ 
we must have $z_{k+1,j} \geq z_{k+1,i}$, and hence $z_{k,i} > z_{k+1,l}$ for 
all $l \geq i$. We now show that we can assume $i = k$.

Suppose $i < k$ and let $w \in \ZZ^\mu_0$ be the element given by
\begin{align*}
w_l &=
\begin{cases}
	(0^{(l)}) & k+2 \leq l \leq n \\ 
	(0^{(i-1)}, -1^{(k-i+2)}) & l = k+1 \\
	(0^{(i)}, -1^{(k-i)}) & l = k \\
	(-1^{(l)}) & 1 \leq l \leq k-1
\end{cases} 
\end{align*} 
where $m^{(j)}$ denotes a sequence of $j$ consecutive entries equal to $m$. 
Then $\Omega^-(w) \subset \Omega^-(z)$ and hence by Lemma 
\ref{L:omega-contained} we have $D_e(w) \in M(-\tilde\rho)$. Applying Lemma 
\ref{L:zero-nonzero} twice to rows $k+1, k$ with $m = -1$, we see that $z' = 
w - \delta^{k+1,k+1} - \delta^{k+1,k}$ also belongs to the support of 
$M(-\tilde\rho)$. Direct inspection shows that $\Omega^-(-\delta^{k+1,k}
-\delta^{k+1,k+1}) \subset \Omega^-(w')$, so $z'' = -\delta^{k+1,k}
-\delta^{k+1,k+1}$ also lies in the support of $M(-\tilde\rho)$, or 
equivalently $D_e(z') \in M(-\tilde \rho)$.

Now any element in the Gelfand-Tsetlin component of weight $z''$ has Cartan 
weight $\tilde\rho -2 \alpha$, where $\alpha$ is the Cartan weight of 
$E_{k+2,k+1}$. The only such elements in $M(-\tilde\rho)$ are the scalar 
multiples of $E_{k+2,k+1}^2 D_e(\mathbf 0)$. Applying Lemma 
\ref{L:zero-nonzero} twice to $D_e(\mathbf 0)$ we see that the projection of 
this element to the $z''$-Gelfand-Tsetlin weight component is zero. This is a 
contradiction, which arose from assuming the existence of $z$ in the support 
of $M(-\tilde\rho)$ but not in $\overline \P(\Omega^+(v))$.
\end{proof}

\paragraph
\about{The essential support of $M(-\tilde\rho)$}
Let $w \in \DD(\mathbf 0)$ be the element given by $w_{n,i} = 0$ and $w_{k,i} 
= w_{k+1, k+1} - i + 1$ for all $k \leq n$. For example if $n = 5$ then
\\

\begin{tikzpicture}
\node (z) at (-2.5,1.25) {$w = $};

\node (51) at (-1.5,2) {$0$};
\node (52) at (-0.5,2) {$0$};
\node (53) at (0.5,2) {$0$};
\node (54) at (1.5,2) {$0$};
\node (55) at (2.5,2) {$0$};

\node (41) at (-1,1.5) {$0$};
\node (42) at (0,1.5) {$-1$};
\node (43) at (1,1.5) {$-2$};
\node (44) at (2,1.5) {$-3$};

\node (31) at (-.5,1) {$-3$};
\node (32) at (0.5,1) {$-4$};
\node (33) at (1.5,1) {$-5$};

\node (22) at (0,0.5) {$-5$};
\node (21) at (1,0.5) {$-6$};

\node (11) at (0.5,0) {$-6$};
\end{tikzpicture}

\medskip

Clearly, $\Omega^+(w) = \Omega^+(\mathbf 0)$ and hence 
$$
\dim M(-\tilde\rho)[w] = |S_\pi| = (n-1)!(n-2)! \cdots 2! 1!,
$$
thus showing $w$ lies in the essential support of $D_e(\mathbf 0)$. In fact, 
we have even a stronger statement, namely, that the essential support of 
$M(-\tilde\rho)$ is equal to the rational cone $\overline \P(\Omega^+(w))$. 
As in our study of the support, we need a preliminary lemma; again, we denote
by $m^{(j)}$ a sequence of $j$ consecutive entries equal to $m$.
\begin{Lemma}
\label{L:pre-essential}
Let $1 \leq s < k \leq n-2$, and let $z \in \DD(\mathbf 0)$ be such that 
$z_{k+1} = (0^{(k)}, -1)$ and $z_k = (0^{(s-1)}, -1^{(k-s+1)})$. If $z$ 
lies in the essential support of $M(-\tilde\rho)$ then so does $z + 
\delta^{k,s}$.
\end{Lemma}
\begin{proof}
Set $z' = z + \delta^{k,s}$, and set $I = \interval{s+1,k}_{(k)}$. We put
$\omega = \omega_0^{z'}$ and $\sigma = \omega (s \ s+1)_{(k)} \beta(I)$. An 
explicit calculation shows that
\begin{align*}
\sigma_{(k)}(i)
	&= 
	\begin{cases}
		i + k - s -1 & \mbox{ if } 1 \leq i < s;\\
		k & \mbox{ if } i = s; \\
		i - s & \mbox{ if } s < i < k; \\
		k - 1 & \mbox{ if } i = k; 
	\end{cases}
\end{align*}
while $\sigma_{(l)} = \omega_{(l)}$ for $l\neq k$, so $\sigma$ is a 
$z$-shuffle. By definition, $\sigma \alpha(I) = \omega (s \ s+1)$, thus 
$\D_{\omega,\sigma\alpha(I)} = \partial_{k,s} - \partial_{k,s+1}$. 

Since $z$ lies in the essential support of $M(-\tilde\rho)$, we know that 
$D_\sigma(z) \in M(-\tilde\rho)$. Using the formulas in Theorem 
\ref{T:gt-big-module}, we see that the coefficient of $D_\omega(z')$ in 
$E_{k,k+1} D_\sigma(z)$ is $(\partial_{k,s}- \partial_{k,s+1})(e_I)$ evaluated 
at $z$. Putting $\tilde e_I = \frac{e_I}{x_{k,s} - x_{k+1,k+1}}$ then 
\begin{align*}
\mathfrak D_{\omega, \sigma\alpha(I)}(e_I)
&=(x_{k,s} - x_{k+1,k+1}) \left(
	\partial_{k,s}-\partial_{k,s+1}
\right)(\tilde e_I) 
- \tilde e_I,
\end{align*}
which evaluates to $(-1)^{k-s+1}$ at $z$. By Proposition 
\ref{P:gt-weight-spaces}(\ref{i:generator}) $p_{z'}(E_{k,k+1}D_\sigma(z))$ 
generates $V(T(\mathbf 0))[z']$ and hence $z'$ lies in the essential support 
of $M(-\tilde\rho)$.
\end{proof}

\begin{Theorem}
\label{thm-esssup Verma}
Let $(k,i) \in \Sigma'$, and denote by $c^{k,i}$ the element in $\ZZ^\mu_0$ 
for which $c^{k,i}_{l,j}$ is $0$ if $l > k$ or if $l = k$ and $i < j$, and 
$-1$ otherwise. Then $\P(\Omega^+(\mathbf{0})) = \sum_{1 \leq i \leq k \leq 
n-1} \NN c^{k,i}$, and $\P(\Omega^+(\mathbf{0}))$ coincides with the essential 
support of $M(-\tilde\rho)$. 
\end{Theorem}
\begin{proof}
The description of $\P(\Omega^+(\mathbf{0}))$ given in the proof of 
Theorem \ref{T:essupp}, so we just have to show that it is equal to the 
essential support of $M(-\tilde\rho)$.

We note first that the edges of $\Omega = \Omega^+(\mathbf 0)$ are the 
edges of the form $[k+1,j] - [k,i]$ for $j \in \interval{k+1}$ and $i \in 
\interval{k}$. Take $z \in \DD(\mathbf 0)$. Then $z$ lies outside 
$\P(\Omega)$ if and only if $\Omega^-(z)$ has at least one edge, say 
$[k,i] - [k+1,j]$, or equivalently if and only if $z_{k,i} > z_{k+1,j}$. Since 
$z_{k,i} \leq 0$ for all $k \leq n-1$, we must have $k+1 < n$. Furthermore, 
since $z$ is in normal form we have $z_{k,l} \geq z_{k,i} > z_{k+1,k+1}$, so 
every edge in $\mathcal V(k,i) = \{[k,l] - [k+1,k+1] \mid l \leq i\}$ is an 
edge of $\Omega^-(z)$. By Lemma \ref{L:omega-contained} if $z$ is in the 
essential support of $M(-\tilde\rho)$ then so is any element $w$ such that the 
edges of $\Omega^-(w)$ are contained in $\mathcal V(k,i)$.

Assume there is an element $z \in \DD(\mathbf 0)$ which is in the essential 
support of $M(-\tilde\rho)$ but not in $\P(\Omega)$ and let $k$ be minimal 
such that some $\mathcal V(k,i)$ is contained in the edge set of $\Omega^-(z)$.
Suppose first that $k = 1$. Then $\mathcal V(1,1)$ is the set of edges of 
$\Omega^-(-\delta^{2,2})$, and hence $-\delta^{2,2}$ lies in the essential 
support of $M(-\tilde\rho)$ and the Gelfand-Tsetlin component of this weight 
has dimension $2$. The Cartan weight of any element in this component is 
$-\alpha - \tilde\rho$, where $\alpha$ is the Cartan weight of $E_{2,3}$, and 
since $M(-\tilde\rho)$ is a free $U^-$-module and a highest weight module with 
highest weight $-\tilde\rho$, it contains only one element with this weight, 
namely $E_{3,2} D_e(\mathbf 0)$, a contradiction. Hence $k \geq 2$.

Let $(k,s) \in \Sigma$ and denote by $z(k,s)$ the element given by
\begin{align*}
z(k,s)_l
	&= \begin{cases}
	 (0^{(l)}) & \mbox{ for } k+2 \leq l \leq n;\\
	 (0^{(k)}, -1) & \mbox{ for } l = k+1; \\
	 (0^{(s)}, -1^{(k-s)}) & \mbox{ for } l = k; \\
	 (-1^{(l)}) & \mbox { for } 1 \leq l \leq k-1; 
	\end{cases} 
\end{align*} 
Applying $k-s$ times Lemma \ref{L:pre-essential} to rows $k, k+1$
we see that $z(k,k-1)$ also lies in the essential support of $M(-\tilde\rho)$. 
Now if $k > 2$ then we can also apply Lemma \ref{L:pre-essential} to $z(k,k-1)$
in rows $k-1, k$ to obtain an element $z'$ in the essential support of 
$M(-\tilde\rho)$ such that $z'_{k-1,1} > z'_{k,k}$, contradicting the 
minimality of $k$. Thus we must have $k = 2$ and $z(2,1) = - \delta^{3,3} - 
\delta^{2,2} - \delta^{1,1}$ must lie in the essential support of $M(-\tilde
\rho)$. In particular this implies that the component of Gelfand-Tsetlin 
weight $z(2,1)$ of $M(-\tilde\rho)$ must be of dimension $6$.

On the other hand, any vector of Gelfand-Tsetlin weight $z(2,1)$ must have
Cartan weight $-\alpha_{1,4} -\tilde\rho$, where $\alpha_{1,4}$ is the Cartan 
weight of $E_{1,2}E_{2,3}E_{3,4}$. However the space of weight $-\alpha_{1,4}$ 
in $U^-$ is four dimensional, which implies that the dimension of the 
Gelfand-Tsetlin component of weight $z(2,1)$ of $M(-\tilde\rho)$ can be at 
most four, a contradiction.
\end{proof}

\begin{Example}
If $n = 4$ then the support of $M(-\tilde\rho)$ consists of all tableaux
\\

\begin{tikzpicture}
\node (41) at (-1,.5) {$0$};
\node (42) at (0,.5) {$0$};
\node (43) at (1,.5) {$0$};
\node (44) at (2,.5) {$0$};
\node (31) at (-.5,0) {$a$};
\node (32) at (0.5,0) {$b$};
\node (33) at (1.5,0) {$c$};
\node (22) at (0,-0.5) {$d$};
\node (21) at (1,-0.5) {$e$};
\node (11) at (0.5,-1) {$f$};

\node () at (3.5,-0.25) {such that};

\node () at (6,.5) {$0 \geq a \geq b \geq c$};
\node () at (6,0) {$d \geq e$};
\node () at (6,-.5) {$a \geq d \geq f$};
\node () at (6,-1) {$b \geq e$};
\end{tikzpicture}

On the other hand, the essential support of $M(-\tilde\rho)$ consists of those tableaux such that $0 
\geq a \geq b \geq c \geq d \geq e \geq f$.
\end{Example}

We conclude this paper with the observation that the simple Verma modules of 
$\mathfrak{gl}(n,\CC)$ ($n\geq 4$), considered as Gelfand-Tsetlin modules, do 
not have necessarily a basis of derivative tableaux.
\begin{Remark} Let $z = - \delta^{1,1} - \delta^{2,2} - \delta^{3,3}$. 
A consecutive application of the Gelfand-Tsetlin formulas shows that
\begin{align*}
E_{4,3} E_{3,2} E_{2,1} D_e(\mathbf 0) 
	&= -D_e(z) + D_{(32)_{(3)}}(z) - D_{(123)_{(3)}}(z) + \\
		& \qquad D_{(12)_{(2)}}(z)- D_{(12)_{(2)}(32)_{(3)}}(z).
\end{align*}
As mentioned above, the Gelfand-Tsetlin component of $M(-\tilde\rho)$ of 
weight $z$ has dimension at most $4$. Hence the Verma module $M(-\tilde\rho)$ 
does not have a basis formed by derivative tableaux. It is an interesting open 
question to find a basis of the Verma modules, even of $M(-\tilde\rho)$, 
inside $V(T(\vv))$.
\end{Remark}

\end{document}